\documentclass[12pt]{amsart}
\usepackage{amsmath}
\usepackage{amssymb}
\usepackage{amsfonts}
\usepackage{amsthm}
\usepackage{verbatim}
\usepackage{amscd}
\usepackage{cite}
\usepackage{leftidx}
\usepackage{enumerate}
\usepackage{txfonts}
\usepackage{manfnt}
\usepackage{amscd}
\usepackage[mathscr]{eucal}
\usepackage{hyperref}
\usepackage{datetime2}
\usepackage{mathpazo}
\newtheorem{thm}{Theorem} % [section]

\newtheorem*{thm*}{Theorem}
\newtheorem*{prop*}{Proposition}
\newtheorem{cor}[thm]{Corollary}
\newtheorem*{cor*}{Corollary}

\newtheorem{lem}[thm]{Lemma}
\newtheorem*{lem*}{Lemma}

\newtheorem*{claim*}{Claim}
\newtheorem{prop}[thm]{Proposition}

\newtheorem{defn}[thm]{Definition}%number together with conj
\theoremstyle{remark}

\newtheorem{rem}[thm]{Remark}
\newtheorem*{rem*}{Remark}
\newtheorem{crit-rem}[thm]{Critical remark}

\newtheorem{example}[thm]{Example}
\newtheorem*{example*}{Example}

\newtheorem*{defn*}{Definition}

\def\inv{^{-1}}

\DeclareMathOperator{\del}{\partial}

\def\refp #1.{(\ref{#1})}

\newcommand\ceil [1] {\lceil #1 \rceil}

\newcommand{\A}{\mathcal{A}}

\newcommand{\kk}{\mathbf{k}}

\newcommand{\ul}[1]{\underline {#1}}

\def\sbr #1.{^{[#1]}}
\def\sfl #1.{^{\lfloor #1\rfloor}}

\def\inv{^{-1}}
\def\?{{\bf{??}}}

\def\A{\Bbb A}

\def\P{\mathbb P}

\def\Z{\mathbb Z}

\def\O{\mathcal O}

\def\rk{\text{rk}}

\def\g{\mathfrak g}

\def\m{\mathfrak m}

\def\1/2{\frac{1}{2}}

\def\I{\mathcal{ I}}

\def\2{{[2]}}
\def\l{\ell}
\def\nl{\newline}

\def\<{\langle}
\def\>{\rangle}

\def\2{{[2]}}
\def\l{\ell}

\def\scl #1.{^{\lceil#1\rceil}}
\def\spr #1.{^{(#1)}}
\def\sbc #1.{^{\{#1\}}}

\def\subpr#1.{_{(#1)}}

\def\beq{\begin{equation*}}
\def\eeq{\end{equation*}}

\newcommand{\llog}[1]{\langle\log {#1} \rangle}

\newcommand{\mlog}[1]{\langle-\log {#1} \rangle}
\def\g3{{\Gamma\spr 3.}}

\newcommand{\eqspl}[2]{
\begin{equation}\label{#1}
\begin{split}
#2\end{split}\end{equation}}

\newcommand{\exseq}[3]{
0\to #1\to #2\to #3\to 0
}

\newcommand{\beginalphaenum}{
\begin{enumerate}\renewcommand{\labelenumi}{ }
\item \begin{enumerate}
}

\def\eex{\end{rm}\end{example}}
\begin{document} 

%\end{document}
%\title{Differential complexes, degeneracy\\ and Hodge theory on Poisson manifolds}
%\pretitle{App}
\title{ Interpolation of curves on Fano hypersurfaces}
%\posttitle{Appendix}
\author %{author}
{Ziv Ran}
%\\ {M\lowercase{ath} D\lowercase{epartment}, UC R\lowercase{iverside}}  }
%\affil{UC Riverside}
%\end{document}
%\Large

%\thanks{\raggedright{

%Partially supported by NSA Grant MDA904-02-1-0094} }
\thanks{arxiv.org 2201.09793}
\date {\DTMnow}% \enddate

%\affil {University of California, Riverside}

%\address {\nl UC Math Dept. \nl
%Skye Surge Facility, Aberdeen-Inverness Road
%\nl
%Riverside CA 92521 US\nl 
%ziv.ran @  ucr.edu\nl
%}

%\email {ziv.ran @ucr.edu}
 \subjclass[2010]{14n25, 14j45, 14m22}

\begin{abstract}
	On a general  hypersurface  of degree $d\leq n$
	in $\P^n$ or $\P^n$ itself,
	we prove the existence of curves  of any genus and high enough degree
	depending on the genus passing
	through the expected number $t$ of general points or incident to a general
	collection of subvarieties of suitable codimensions. In some cases
	we also show that the family of curves through $t$ fixed points
has general  moduli as family of $t$-pointed curves. These results imply
positivity of certain intersection numbers on Kontsevich spaces of stable
maps. An arithmetical appendix by M. C. Chang descibes the set of numerical
characters ($n, d$, curve degree, genus) to which our results apply.
\end{abstract}
\maketitle
%\footnote*{With an appendix by M. C. Chang}
%\tiny
\tableofcontents
\normalsize
\section*{Introduction}
\subsection{Notions of interpolation}
A curve $C$ on a variety $X$ is said to be \emph{interpolating}
or to have the \emph{interpolation property} if  $C$ can 
be deformed so as to go through the expected number of general points on $X$.
Here 'expected number'
means, in terms of the normal bundle $N=N_{C/X}$,  
the largest integer $t$ such that $(n-1)t\leq\chi(N), n=\dim (X)$ or explicitly,
where $g$ denotes the genus of $C$,
\[t=[s(N)]+1-g=[\frac{C.(-K_X)+2g-2}{n-1}]+1-g.\]
This makes most sense if $H^1(N)=0$, so that $C$ moves in an unobstructed
family of the expected dimension, i.e. $h^0(N)$. The adjective 'separable' may be added if
the appropriate correspondence is separable over the symmetric product $X^{(t)}$,
which is of course automatic in char. 0.\par
A stronger property than interpolation, though equivalent in genus 0, 
is that of \emph{ultra- interpolation}. $C$ is said to be ultra-interpolating if for a sufficiently general collection 
of subvarieties $Y_i\subset X$, $C$ can be deformed so as to meet all of them, provided
\[\sum(\mathrm{codim}(Y_i)-1)t\leq\chi(N).\]
The existence of an interpolating or ultra-interpolating
curve implies positivity of certain intersection numbers on Kontsevich spaces
of stable maps, which measure the 'virtual' number of such curves.
\par 
Another property related to interpolation is that of \emph{modular interpolation}.
Given $m$ fixed general points on $X$, the family of deformations of $C$
going through them yields a family of $m$-pointed curves of genus $g$
and one may inquire whether a general member has general moduli as such.
When this holds for all $m$ up to the expected number, namely 
 \eqspl{t-eq}{t=[\chi(T_X|_C)/n]=[(-C.K_X)/n]+1-g,}   we will say that $C$
is \emph{moduli-interpolating}. Again the adjective 'separable' may be added if
the appropriate map to the moduli of $t$-pointed curves is separable.
Again there is an ultra version.\par
The various separable interpolation properties of a curve $C$ 
are  equivalent to
certain properties called balancedness or ultra-balancedness of either the 
normal bundle $N$ or the restricted ambient tangent bundle
$T=T_X|_C$. Thus separable balance is equivalent to the
property  that for a general effective divisor $D_t$ of degree $t$ on $C$ one has
either $H^1(N(-D_t))=0$ or $H^0(N(-D_t))=0$. Separable ultra balance
is equivalent to the property that for any subsheaf $N^U\subset N$ such that
$N/N^U$ is a locally free $\O_{D_t}$-module, one has either $H^1(N^U)=0$
or $H^0(N^U))=0$. Separable modular interpolation means that for $t$ as in \eqref{t-eq},
one has $H^1(T(-D_t))=0$. It is via these bundle properties that we will
approach interpolation. \par 
\subsection{Known results}
There is a fair amount of work on curve interpolation
 in the case where $C$ is rational and $X$ is a Fano manifold,
e.g. $\P^n$, a Fano hypersurface in $\P^n$ or a Grassmannian, starting with
the case of rational curves in $\P^n$, due to Sacchiero
\cite{sacchiero}; see \cite{eisenbud-vandeven-normal},   \cite{coskun-riedl}, 
\cite{caudatenormal}\cite{scroll}\cite{ran-normal}\cite{hypersurf}.
For curves of higher genus and $X=\P^n$, there are older results for elliptic curves due to
Ellingsrud and Laksov \cite{elligsrud-laksov},
Hulek \cite{hulek-elliptic} and Ein and Lazarsfeld \cite{ein-laz-normal}, and for
$n=3$ due to Perrin \cite{perrin}. More recently,
  comprehensive interpolation  results for $X=\P^n$, any $n$,  were obtained  by
A. Atanasov, E. Larson and D. Yang \cite{alyang}, who showed that
a general nonspecial
 curve of any genus  is interpolating.
To my knowledge there are no results in the literature on interpolation,
much less ultra-interpolation,  for higher-genus curves 
and ambient spaces other than $\P^n$.\par
As for modular interpolation, in case $X=\P^n$, $g=0$ and any $e\geq n$, it is
easy to see that any sufficiently general rational curve of degree $e$ is ambient-balanced.
But already for
$X$ a Grassmannian, $g=0$ and 'most' degrees $e$, there are no moduli-
interpolating curves of degree $e$ (see Example \ref{ambient-example}).
Thus for 'most' varieties $X$ one would expect some topological obstructions
in terms of degree and genus in order for a curve to be ambient-balanced.

%\par 
\subsection{New results}In this paper we consider separable interpolation,
ultra interpolation and modular 
interpolation in arbitrary genus
on $\P^n$ and on general Fano hypersurfaces, i.e. hypersurfaces 
$X$ of degree $\leq n$
in $\P^n, n\geq 4$. Notably, we will show:
\begin{itemize}\item
\par  (See \S \ref{pn-sec}) In $\P^n$, the general curve of genus $g$ 
and degree $e\geq 2(g+1)n$,
 is ultra-interpolating and ultra ambient-interpolating
(see Corollary \ref{ambient-cor}).
\item
(See \S \ref{antican-sec}) On a general hypersurface of degree $n$ in $\P^n$, 
$n\geq 4$, there exist 
ultra-balanced, ultra ambient-balanced curves of genus $g$ and degree $e$
provided either $g\geq 1$ and 
$e\geq 4g(n-1)$ or $g=0$ and $e\geq n-1$.

%separably 
% interpolating and moduli-interpolating curves of any genus $g\geq 1$
%and degree $e\geq 4g(n-1)$ and genus 0 and degree $e\geq n-1$ (see \S \ref{antican-sec}, especially Corollary
%\ref{moduli-antican-cor});
%for $e\geq 2(g+1_n(n-1)$, the curve may be assumed ultra-balanced
%(see Corollary \ref{ultra-anticanonical-cor}.
\item
(See \S \ref{fano-sec})
On a general hypersurface of degree $d<n$ in $\P^n$, 
there exist balanced (resp. ambient-balanced) curves of any genus $g\geq 0$ and degree
$e$ provided $(n,d,g, e)$ satisfy certain arithmetical conditions. 
%which comes from some slope-
%matching conditions that arise in trying to lift a balanced curve from the base of a fibration
%to its total space.
An  arithmetical appendix by M. C. Chang gives sufficient conditions
for these conditions to hold, showing in particular that for given $n, d, g$, 
the conditions for balance (resp. ambient balance) hold for all $e$ in
at least one arithmetic progression with difference $d(n-2)$
(resp. for infinitely many $e$) (see
Theorem \ref{fano} and the ensuing examples).
\end{itemize}
%******
%We will construct separably interpolating and moduli-interpolating curves of every  genus and arbitrarily high degree 
%on $X$, see Theorem \ref{bmb-antican-thm} and Corollary \ref{moduli-antican-cor}
%(for $d=n$) and
%Theorem \ref{fano} (for $d<n$). In the case $d=n$, all sufficiently large curve degrees
%(depending on the genus) are obtained.
%In fact the curves we construct  have a stronger property called
%balancedness signifying that deformations of the curve through a fixed
%collection of general points fills out a variety of the expected dimension.
%These results extend those of \cite{caudatenormal} for genus 0.
%We will also construct such higher-genus balanced and ambient-balanced, hence
%moduli-interpolating curves in $\P^n$,  see 
%Corollary \ref{Pn} and Corollary \ref{ambient-cor}.
%*************
\subsection{Methods}
The method of proof builds on the one used before in \cite{caudatenormal} 
to prove balancedness for rational curves, and is likewise
based on fans and fang degenerations, degenerating the curve together with its ambient space, be it 
$\P^n$ or a hypersurface (which in turn degenerates together with its own ambient $\P^n$)
to a reducible pair. 
More specifically, we consider flags of the form
\[C_1\cup C_2\subset X_1\cup X_2\subset P_1\cup P_2\]
where $P_1$ and $P_2$ are blowups
\[P_1=B_{\P^m}\P^n, P_2=B_{P^{n-m-1}}\P^n\]
glued along the exceptional divisor $\P^{n-m-1}\times\P^m $, $X_1\cup X_2$ is a suitable Cartier
divisor on $P_1\cup P_2$ (e.g. in the proper fang case $0<m<n-1$, $X_1, X_2$
are birational transforms of hypersurfaces of degree $d$ with multiplicity $e$ (resp. $d-e$)
on $\P^m$ (resp. $\P^{n-m-1}$)); and $C_1\cup C_2$ is a lci curve on $X_1\cup X_2$.
Then the inclusion $X_1\cup X_2\subset P_1\cup P_2$ smooths to an inclusion $X\subset\P^n$
of a smooth hypersurface of degree $d$.  It can be shown
that under suitable conditions on normal bundles the
inclusion $C_1\cup C_2\subset X_1\cup X_2$
smooths to an inclusion
$C\subset X$ of  a smooth curve.
To construct good curves $C\subset X$ one is thus reduced to constructing
'good'-in a suitable sense- curves $C_1\subset X_1, C_2\subset X_2$. This is the method
used in \cite{caudatenormal} and here extended to higher genus and to ambient and ultra balancedness. 
\subsection{Contents}
 Elementary properties of balanced 
and ultra-balanced bundles are developed in \S \ref{balanced-sec}.
In \S \ref{relative-sec} we study a relative version of the tangent bundle
for a family of varieties degenerating to normal-crossing double points.
This is useful in studying moduli-interpolating families. The contents of \S\S \ref{pn-sec}, \ref{antican-sec},
\ref{fano-sec} have been described above. The Appendix by M. C. Chang studies the
roundup equations that arise mainly as one tries to construct balanced bundles as extensions,
such as those that occur in studying curves in a fibration, trying to lift a good (e.g. balanced)
curve in the base to one
in the total space.\vskip .5cm
%The method could conceivably extend to hypersurfaces of degree $<n$, but some difficulties
%remain.
{\bf{Acknowledgment}} I am grateful to
M. C. Chang for providing the Appendix,
as well as Example \ref{d=n-1-example},
to the referee for helpful comments,
and to R. Lazarsfeld and L. Ein for helpful references.

\section{Balanced bundles in any genus}\label{balanced-sec}
We work over an algebraically closed field or arbitrary characteristic.\par
\subsection{Basics}
Let $E$ be a vector bundle of slope $s=s(E)$ on a curve $C$ of genus $g$.
We set \[t(E)=s+1-g=\frac{\chi(E)}{\rk(E)}\] 
and call it the \emph{Euler slope} or e-slope of $E$. 
Also let \[r(E)=\deg(E)\%\rk(E)=\chi(E)\%\rk(E)\] where $\%$ denotes remainder;
this is called the \emph{remainder} of $E$.\par
For an effective divisor $D$ on $C$ we denote by $\rho_D$ the restriction map
\[\rho_D:H^0(E)\to H^0(E\otimes\O_D).\]
If $D$ is general of degree $t$ we will denote $\rho_D$ by $\rho_t$.
Here 'general' means, in case $C$ is reducible, general in some component of $C^{(t)}$.
\begin{defn} A bundle $E$ is said to be \emph{regular} if $H^1(E)=0$.\par
	$E$ is \emph{semi-balanced}
if\par (i) $E$ is generically generated; \par
(ii) $E$ is regular;\par
(iii) the restriction map
$\rho_t$
is surjective for all $t\leq t(E)$.\par
A semi-balanced bundle is \emph{balanced} if $\rho_t$ is moreover injective
for all $t\geq t(E)$.\par
 A balanced bundle is \emph{perfectly balanced} if in addition
$s$ is an integer.\end{defn}
The notion of balanced bundle can be generalized as follows.
\begin{defn}
Let $E$ be a regular, generically generated bundle. 
Given a weight vector $\ul u=(u_1,...,u_t), 0\leq u_i\leq\rk(E)$, $E$ is said to be \emph{$\ul u$-balanced} 
if there exist points
$x_1,...,x_t$, each general in some component of $C$, 
and for each $i$, a general skyscraper quotient $U_i$ of $E|_{x_i}$ of dimension $u_i$,
such that the restriction map
\[\rho_{\ul u}: H^0(E)\to H^0(\bigoplus U_i)\]
has maximal rank. $E$ is \emph{ perfectly $\ul u$-balanced} if $\rho_{\ul u}$
	is an isomorphism.\par 
$E$ is said to be \emph{ultra-balanced} if it is $\ul u$-balanced for every $\ul u$.
\qed\end{defn}
Obviously $\rho_t$ is just $\rho_{\rk(E),...,\rk(E)}$, so $E$ is balanced
iff it is $\ul u$-balanced for all scalar weight-vectors of the form
$(\rk(E),...,\rk(E))\in\Z^t, \forall t$.
Note that for $E$ regular, $\rho_t$ can be surjective only for $t\leq t(E)$. Also, note that in the definition,
we are requiring $U_i$ to be killed by the maximal ideal $\m_{x_i}$ rather than just some power of it.
\begin{rem} Regarding balancedness vs. (semi) stability.
	For a bundle of slope $s$ on a curve of genus $g$, balancedness excludes subbundles
	of degree $s+1-g$ or less while stability excludes subbundles of degree $s$ or less.
	Thus balancedness seems not implied by stability if $g>1$ though we don't have an
	explicit example of an unbalanced stable bundle. Conversely there exist direct sums of lines 
	bundles that are ultra-balanced but not stable (see Lemma \ref{sum-ultra-lem}).
	\end{rem}
\begin{lem} Suppose $E$ is generically generated. Then the following are equivalent:
	\par
	 (i) $E$ is semi-balanced;\par (ii) 
	 for general points $x_1,...,x_t\in C$ and $\forall t\leq t(E)$, 
	 we have $H^1(E(-x_1-...-x_t))=0$
or equivalently 
\[h^0(E(-x_1-...-x_t))=\chi(E(-x_1-...-x_t)) ;\]
\par(iii) $h^0(E)=\chi(E)$ and $h^0(E(-x_1-...-x_t))=h^0(E)-t.\rk(E), 
\forall t\leq t(E)$.\par
Moreover, if $E$ is semi-balanced, then $E$ is balanced iff 
$H^0(E(-x_1-...-x_t))=0, \forall t\geq t(E)$.\par
In particular, the condition that $\rho_t$ be injective or surjective depends
only on the linear equivalence class of $\sum x_i$ hence only on $t$ if $g=0$.
\end{lem}
The proof may be left to the reader.\qed\par
\begin{lem}\label{chi}
	A balanced bundle $E$ is ultra-balanced provided $\rho_{\ul u}$ is an isomorphism for all 
	weight-vectors $\ul u$
of weight	$\sum u_i=\chi(E)$.
	\end{lem}
\begin{lem}
	A generically generated bundle $E$ is $\ul u$-balanced iff, in the above notations,
	the modified bundle
	\[E^{\ul u}=\ker(E\to\bigoplus U_i)\]
	has natural cohomology, i.e. $h^0(E^{\ul u})h^1(E^{\ul u})=0$.
	\end{lem}
For rational curves, the above notion of balanced coincides with the usual:
\begin{lem}
	If $g=0$, $E$ is balanced iff $E$ is ultra-balanced iff $E\simeq b_1\O(a+1)\oplus b_0\O(a)$ for some
	$a\geq 0, b_0>0, b_1$.
	\end{lem}
  \begin{proof}
  If $E$ has the form $b_1\O(a+1)\oplus b_0\O(a)$ then so does a general modification of $E$,
  so $E$ is ultra-balanced. Conversely assume $E$ is balanced and let $a$ be the smallest
  degree of a line bundle quotient
  (= summand) of $E$. By semi-balancedness clearly $[s(E)]=a\geq 0, [t(E)]=[a]+1$. If $E$
  has a line bundle summand of degree $\geq a+2$ then $H^0(E(-x_1-...-x_{t+1}))\neq 0$,
  contradicting balancedness. 	
  	\end{proof}
  Note that for $g=0$ the 'test' divisor $\sum x_i$ may actually be an
  arbitrary effective divisor of degree $t$. 
  For general $g$ the injectivity or surjectivity conditions for balancedness 
  depend only on the linear equivalence class of $\sum x_i$.
  Also for general $g$, half the above characterization still holds:
  \begin{lem}
  	Suppose $E$ admits a filtration whose quotients $L_1,...,L_r$ are 
  	 line bundles such that $\deg(L_1),...\deg(L_r)\in[a, a+1]$
  	for some $a\geq 2g-1$.
  	Then $E$ is balanced.
  	\end{lem}
  \begin{proof}
  	If $D_t$ denotes a general effective divisor of degree $t$ then it is easy to check that
  	\[H^1(E(-D_t))=0, t\leq g,\]
  	\[H^0(E(-D_t))=0, t\geq g+1.\]
  	\end{proof}
  \par
  There is a version of this for ultra-balanced:
  \begin{lem}\label{sum-ultra-lem}
  	Let $E$ be a direct sum of line bundles with degrees in $[a, a+1], a\geq 2g-1$.
  	Then $E$ is ultra-balanced.
  		\end{lem}
  	\begin{proof}
  		As has been noted, if $L$ is a line bundle of degree $a\geq 2g-1$ then
  		\[H^1(L(-D_t))=0, t\leq a+1-g,\]
  		\[H^0(L(-D_t))=0, t\geq a+1-g.\]
  		We can write
  		\[E=L_1\oplus...\oplus L_s\oplus L_{s+1}\oplus...\oplus L_r\]
  		where
  		\[\deg(L_i)=\begin{cases}
  			a+1, i\leq s;\\
  			a, i>s
  			\end{cases}
  		\] and the subbundle $L_1\oplus ...\oplus L_s\subset E$ is uniquely determined.
  		Then we have $\chi(E)=ra+s$. If $\ul u=(u_1,...,u_t)$ is a weight vector, we have,
  		by generality of the quotient involved,
  		\[
  		E^{u_1}=  			L_1(-p)\oplus...\oplus L_{u_1}(-p)\oplus L_{u_1+1}\oplus...\oplus L_r,
  			  		\] where $p\in C$ is a general point, and this
  			  		is a direct sum of line bundles of degrees in $[a, a+1]$ if $u_1\leq s$
  			  		or $[a-1,a]$ if $u_1\geq s$.
  			  		Then it is easy to check, e.g.
  		by induction of the length of the weight-vector $\ul u$, that
  		\[H^1(E^{\ul u})=0, |\ul u|\leq\chi(E),\]
  		\[H^0(E^{\ul u})=0, |\ul u|\geq\chi(E).\]
  		\end{proof}
  We can similarly characterize semi-balanced bundles on $\P^1$:
  \begin{lem}
  	A globally generated bundle of slope $s$ on $\P^1$ is semi-balanced iff the smallest
  	degree of its line bundle summands is $[s]$.\qed
  	\end{lem}
  \begin{example}
  	The bundle $\O(2)\oplus 2\O$ on $\P^1$ is semi-balanced but not balanced.\par
  	\end{example}
  There is a partial extension for elliptic curves:
  \begin{lem}
  	Assume $g=1$, $E$ is generically generated and regular,  and  and that $E$ is 
  	either (1) poly-stable  or (2) semi-stable of non-integer slope.
  	Then $E$ is balanced. 
  	\end{lem}
  \begin{proof}
  	Here $t(E)=s(E)$ and for $t\leq t(E)$ (resp. $t\geq t(E)$), $E(-x_1-...-x_t)$
  	has nonnegative (resp. nonpositive) slope so the conclusion is immediate.
  	\end{proof}
  For general $g$ one might conjecture that if $E$ is regular and generically
  generated then $E$ is balanced iff the slopes of its Harder-Narasimhan graded
  pieces are all in some length-1 interval. 
  \subsection{Splitting, modifying and matching}
  The following result is useful in constructing some semi-balanced and sometimes
  balanced bundles by smoothing from a bundle on a reducible curve. 
  \begin{lem}\label{reducible}
  	Let $C=C_1\cup C_2$ be a nodal curve such that $C_1\cap C_2$ consists of 
  	$k$ general points 
  	on $C_1$. Let $E$ be a  bundle
  	on $C$. Assume\par (i) $E$ is regular and generically generated;\par
  	 (ii) $E_i=E_{C_i}$ are balanced, $i=1,2$;\par (iii) the remainders
  	satisfy $r(E_1)+r(E_2)<r(E)$ (e.g.  $E_{C_1}$ or $E_{C_2}$ is perfectly balanced);\par 
  	(iv) $t(E_1)\geq k$.\par
  	 Then\par  (a) $E$ is semi-balanced.\par (b) Moreover if $r(E_2)=0$, $E$ is balanced.
  	\end{lem}
  \begin{proof}
  	The respective genera satisfy $g=g_1+g_2+k-1, k=C_1.C_2$ hence for the Euler slopes
  	\[t(E)=t(E_1)+t(E_2)-k.\]
  	For $t=[t(E)]$ write $t=t_1+t_2$ where \[t_1=[t(E_1)]-k, t_2=[t(E_2)].\]
  	To prove $E$ is semi-balanced, choose general points  
  	\[x_{11},...,x_{1t_1}\in C_1, x_{21},...,x_{2t_2}\in C_2.\]
  	By balancedness of $E_2$, there is a section $s_2$ of $E_2$ with 
  	arbitrary assigned values
  	at $x_{21},...,x_{2t_2}$. By balancedness of $E_1$ there is a section $s_1$ of $E_1$ with
  	arbitrary assigned values at $x_{11},...,x_{1t_1}$ and matching $s_2$ on $C_1\cap C_2$. Then $s_1$
  	and $s_2$ glue to a section of $E$ with assigned values at all the $x_{ij}$.
  	This proves (a).  Then the proof of (b)
  	is similar.
  	\end{proof}
  \begin{rem*}Note the absence of a 'general gluing' assumption over $C_1\cap C_2$. The result will be used mainly
  in case $E_2$ is perfectly balanced.\end{rem*}
The same argument also proves:
\begin{lem}
	Let $C=C_1\cup C_2$ be a nodal curve such that $C_1\cap C_2=\{p_1,...,p_k\}$ 
	consists of $k$ general points on each component..
	Let $E$ be a regular, rank-$r$ bundle on $C$ and 
	$\ul u, \ul v$ weight-vectors.
	Assume:\par (i) $E_{C_1}$ is $\ul u$-balanced;\par
	(ii) $E_{C_2}$ is $\ul v$- balanced;\par
	(iii) The restriction map $H^0(E|_{C_1}^{\ul u})\oplus H^0(E_{C_2}^{\ul v})\to H^0(E|_{p_1,...,p_k})$
	is surjective\par
	Then $E$ is $(\ul u,\ul v)$-balanced.
	\end{lem}
\begin{proof}
	$H^0(E|_{C_1}(-\sum p_i))\oplus H^0(E|_{C_2}(-\sum p_i))$ is a subspace of $H^0(E)$
	which already surjects onto $H^0(U_1\oplus...\oplus V_t)$.
	\end{proof}
	The following property of ultra-balanced bundles is immediate from the
definition but worth noting:
\begin{lem}
	Let $E$ be an ultra-balanced bundle and $E'=E^u\subset E$ a general down modification,
	i.e. kernel of a general surjection $E\to \bigoplus u_i\kk_{p_i}$, 
	such that $E'$ is regular and generically generated. Then $E'$ is ultra-balanced.
	In particular, if $D_t=\sum\limits_{i=1}^t p_i$ is a general effective divisor
	and $E(-D_t)$ is regular and generically generated, then $E(-D_t)$ is ultra-balanced.	
\end{lem}
  The following two lemmas, which are analogues of simple facts  in the case of rational curves, show that a general
  (up or down) elementary modification of a balanced bundle is balanced:

%******** 

%************** 

	\begin{lem}\label{down}
		Let $E$ be a balanced bundle and $E'\subset E$ a general locally corank-1 modification
		at some general points. Assume $E'$ is regular and generically 
		generated.  Then $E'$ is balanced.
		\end{lem}
	\begin{proof} 
		It suffices to prove this for modification at a single point $p$, so $E'\subset E$
		is the kernel of a general surjection $E\to\kk_p$. 
		Now if $t(E)<1$, the conclusion is obvious, so assume $t(E)\geq 1$.
		We first prove $E'$ is semi-balanced.
		Let $t=[t(E)]>0$.
		Assume first $E$ is not perfect. This easily implies that $[t(E')]=t$. 
		Then for general $x_1,...,x_t$, we get a subsheaf
		\[H^0(E(-x_1-...-x_t))\otimes\O\subset E(-x_1-...-x_t)\]
		that is not contained in the kernel of the (general) modification at $p$.
		Hence $H^0(E'(-x_1-...-x_t))$ has the expected dimension so that
		$H^0(E')\to E'_{x_1,...,x_t}$ is surjective so $E'$ is semi-balanced.\par
		If $E$ is perfect then $t(E')=t(E)-1$, therefore for a general divisor
		$x_1+...+x_{t-1}$, $H^0(E(-x_1-...-x_{t-1}))$ has the expected dimension and
		the restriction map
		\[H^0(E(-x_1-...-x_{t-1}))\to E(-x_1-...-x_{t-1})|_p\]
		is surjective. Therefore the kernel $H^0(E'(-x_1-...-x_t))$ 
		of the restriction map has the expected dimension and
		 semi-balancedness follows.\par
		 Now the injectivity statement required to show $E'$ balanced is obvious if
		 $\ceil{t(E')}=\ceil{t(E)}$. Otherwise, $t:=\ceil{t(E')}=\ceil{t(E)}-1$ and
		 the required injectivity for $E'$ follows from injectivity of
		 $H^0(E)\to E_{x_1,...,x_t, p}$.
		\end{proof}

	There is a similar statement for up modifications:
	\begin{lem}\label{up}
		Let $E$ be a balanced bundle and $E\subset E^+$ a general locally corank-1 modification
		at some general points.   Then $E^+$ is balanced.
		\end{lem}
	\begin{proof} First it is obvious that $E^+$ is regular and generically generated.
		For balancedness,
		it again suffices to prove it for the case of modification at a single point $p$,
		so $(E^+)^*\subset E^*$ is the kernel of a general surjection $E^*\to\kk_p$
and $E_p\to E^+_p$ has kernel a general 1-dimensional subspace.
		Now semi-balancedness
		is obvious if $[t(E)]=[t(E^+)]$. If not, then $t(E^+)=\ceil{t(E)}=[t(E)]+1:=t+1$
		and in particular $t(E^+)$ is an integer.
		Now $H^0(E(-x_1-...-x_t))\subset H^0(E^+(-x_1-...-x_t)$ injects to $E'(-x_1-...-x_t)|_p$
		and its image is just the inverse image of the natural map $E'\to\kk_p$.
		Therefore the kernel of $H^0(E^+(-x_1-...-x_t))\to \kk_p$ is contained in the
		latter image, hence must vanish because $H^0(E(-x_2-...-x_t-p))=0$.
		This proves $H^0(E^+(-x_1-...-x_t))\to E^+_p$ is injective, i.e. surjective,
		so $E^+$ is semi-balanced.\par
		Now to prove $E^+$ is balanced let $t+1:=\ceil{t(E^+)}\geq\ceil{t(E)}$. 
			Then $t(E)<t+1$.
			Now the kernel of $H^0(E^+(-x_1-...-x_t))\to E^+|_p$ corresponds to
			the intersection of the image of $H^0(E(-x_1-...-x_t))\to E|_p$ with the
			 kernel if $E|_p\to E^+|_p$ which is a general 1-dimensional subspace
			 and the intersection is trivial because the latter
			image is a \emph{proper} (maybe trivial) subspace thanks to $t(E)<t+1$. 
			Thus $H^0(E^+(-x_1-...-x_t-p))=0$ so $E^+$ is balanced.
				\end{proof}
	The following Lemma strengthens Lemma 25 of \cite{caudatenormal}
	and generalizes it  to arbitrary genus (note that Cases 2,3 are new even for genus 0):
	\begin{lem}\label{match}
		Let \[\exseq{E_1}{E}{E_2}\] be an exact sequence of vector bundles on
		a curve such that $E_1, E_2$ are balanced of respective slopes
		$s_1, s_2$. Assume either: \par Case 1:
		\[[s_1]=[s_2];\]
		or Case 2: 
		\[s_2=[s_1]+1;\]
		or Case 3:
		\[s_1=[s_2]+1.\]
		Then $E$ is balanced. Moreover  the slope $s=s(E)$ satisfies:
		\par Case 1: $[s]=[s_1]$;\par
		Case 2: $[s]=s_2$;\par
		Case 3:: $[s]=s_1$.
		\end{lem}
	\begin{proof} 
	Apply the Snake Lemma to the following (exact, since $H^1(E_1)=0$) diagram,
	in which $D_m=p_1+...+p_m$ denotes a general effective divisor of degree $m$:
	\eqspl{}{
	\begin{matrix}
		0\to&H^0(E_1)&\to&H^0(E)&\to&H^0(E_2)&\to 0\\
		&\rho_1\downarrow&&\rho\downarrow&&\rho_2\downarrow&\\
		0\to& E_1|_{D_m}&\to&E|_{D_m}&\to&E_2|_{D_m}&\to 0
		\end{matrix}	
	}	
		Case 1:
		The assertion about $s$ is obvious and implies
		 \[t:=[t(E)]=[t(E_1)]=[t(E_2)].\] 
		 Taking $m=t$, we have $\rho_1, \rho_2$ surjective hence so is $\rho$.
		Taking $m=\ceil{t(E)}$, $\rho_1, \rho_2$ are injective hence so it $\rho$.\par
		Case 2: Note this case can occur only if $s_2$, hence $t_2=t(E_2)$ is an integer.
		Taking $m=t_2$, $\rho_2$ is an isomorphism and $\rho_1$ is injective, hence
		$\rho$ is injective. Taking $m=t_2-1$, $\rho_1$ and $\rho_2$ are surjective hence so is $\rho$.
		\par
		Case 3 is similar to Case 2.
		\end{proof}
\subsection{Balanced and ultra-balanced curves, Kontsevich intersections} 
A lci curve $C\to X$ is said to be separably regular or (semi-, perfectly)
balanced if its normal bundle $N_{C/X}$ has the corresponding property. Regularity means that $C$ belongs to
a smooth family of the expected dimension. 
Semi-balance implies (and in char. 0 is equivalent to) the semi-interpolating property,
i.e. that $C$ can be 
deformed  to go through the expected number of general points of $X$, and balance 
implies  moreover that  the 
subvariety of $X$
filled up by the deformations through a fixed maximal collection of general points has the 
expected dimension. 
%We set
%\[s(C)=s(N_{C/X}), t(C)=t(N_{C/X}), r(C)=r(N_{C/X}),\]
%and call them the slope, e-slope and normal remainder of $C$ resp.
When $X$ contains a (semi-) balanced curve we will say that $X$ has the
(semi-) interpolation property (for curves of genus $g(C)$ and degree $\deg(C)$
if understood). \par
If $C$ is reducible and $C_1\subset C$ is a component, we will say $E$
is (semi-) balanced around $C_1$ if $H^1(E)=0$, $E$ is generated by its
sections at a general point of $C_1$, and the required surjectivity
or injectivity statements as appropriate hold for general points of $C_1$.\par
If $C$ has degree $e$ and genus $g$ in $X=\P^n$ then
\[t(C)=e+1-g+[2\frac{e-1+g}{n-1}].\]
In particular if $C$ is nondegenerate (so that  $e\geq n$) and nonspecial
(so that $e+1-g=\chi(\O_C(H))\geq n+1$), we have $t(C)\geq n+3$. 
\par
	See \cite{caudatenormal}, especially \S 1 and \S 5 for various information on normal
	bundles and fangs.\par
	A curve $C\to X$ is said to be ultra-balanced if its normal bundle is. This condition
	has an interesting interpretation in terms of intersection numbers on Kontsevich
	spaces of stable maps. Thus let $M_{g, t}(X)$ be the Kontsevich space of stable
	$t$-pointed maps $C\to X$ where $(C, x_1,...,x_t)$ is a $t$-pointed 
	stable curve of genus $g$. Let
	\[\sigma_i:M_{g, t}(X)\to X, i=1,...,t\]
	be the natural maps. Let $h$ be a birationally ample divisor on $X$  and set
	\[\eta_i=\sigma_i^*(h).\]
	Define
	\[I_M(0, u_1,...,u_t)=\int\limits_M\eta_1^{u_1}...\eta_t^{u_t}.\]
	This definition will shortly be extended to the case of a nonzero first argument.
	\begin{prop}\label{ultra-k}
		Let $M$ be a component of $M_{g,t}(X)$ whose general point has the form
		$(C, x_1,...,x_t)$ where $C$ is ultra-balanced (resp. balanced). 
		Then for all $u_1,...,u_t$
		such that \[u_1+...+u_t=\chi(N_{C/X})=(C.-K_X)+(n-3)(1-g),\] 
		(resp. and such that $u_2=...=u_t=n$) we have
		\[I_M(0,u_1,...,u_t)>0.\]
		\end{prop}
	\begin{proof}
		Considering $X\subset\P^N$, there is a natural map
		\[F:M\to (\P^N)^t.\]
		Our ultra-balanced hypothesis implies that for
		$Z=P^{N-u_1}\times...\times \P^{N-u_t}$,  $F\inv(Z)$ contains an isolated reduced point.
		Therefore the intersection number $F_*(M).Z>0$, which implies our result in the ultra-balanced
		case. The balanced case is similar. 
		\end{proof}
	\subsection{Ambient-balanced curves}
	A curve $C\to X$ of genus $g$ is said to be \emph{ambient-balanced} if the
	restricted tangent bundle $T_X|_C$ is semi-balanced, i.e.
	 for all \[t\leq t(T_X|_C)=(-K_X.C/n)+1-g, n=\dim(X),\]
	and general points $x_1,...,x_t\in C$, we have
	\eqspl{vanishing-ambient}{H^1(T_X|_C(-x_1-...-x_t))=0.}
	 Note that the vanishing \eqref{vanishing-ambient} implies
	$H^1(N_{C/X}(-x_1-...-x_t)=0$ so that a general deformation of $C$ contains
	$t$ general points of $X$. However ambient balance does not imply balance
	because \eqref{vanishing-ambient} is only assumed for $t\leq t(T_X|_C)$
	but usually $t(N_{C/X})>t(T_X|_C)$.\par
Now	\eqref{vanishing-ambient} also implies
	 surjectivity
	  the natural map induced by the normal sequence
	\[H^0(N_{C/X}(-x_1-...-x_t))\to H^1(T_C(-x_1-...-x_t)).\] Consequently
	we have
	\begin{cor}
		If $C\to X$ is  ambient-balanced then 
		 $C$ is separably moduli-interpolating, i.e. for 
			$t\leq (-C.K_X/n)+1-g$ and general points $x_1,...,x_t\in X$,
	the family of deformations of $C$ in $X$ passing through $x_1,...,x_t$
		has separably general moduli as a family of $t$-pointed curves.
		\end{cor}
	Thus, for an ambient-balanced curve $C$ we are able to impose on
	deformations of $C$ simultaneously a fixed set of $t$ general points of $X$ and
	fixed set of $t$-pointed moduli where $t=[-C.K_X/n]+1-g$. Note that such moduli are
	nontrivial even if $g=0$ provided $t\geq 4$. \par
	For genus 0 and $X=\P^n$, it follows easily, e.g. from \cite{caudatenormal}, Lemma 26 
	that a general deformation of any given curve $C$ is ambient-balanced.
	For higher genus, see Corollary \ref{ambient-cor} below.\par
	For example, the rational normal curve in $\P^n$ is both 
	perfectly balanced and perfectly ambient-balanced.
	\begin{example}\label{ambient-example}
		To put matters in perspective consider the case of a Grassmannian $X=G(k,n)$
		with its tautological subbundle $S$ and quotient bundle $Q$
		and tangent bundle $T_X=S^*\otimes Q$.
		For a rational curve $C\subset X$ of degree $e$, it is easy to see
		that on a general deformation of $C$, both $S$ and $Q$ will be balanced but,
		unless $k|e$ or $(n-k)|e$, both will be \emph{imperfect}, hence $T_X|_C$
		will be unbalanced. Consequently, $X$ contains an ambient-balanced rational
		curve of degree $e$ iff
		either $k|e$ or $(n-k)|e$. 
		In particular the set of degrees of ambient-balanced curves in $X$
		constitutes 2 arithmetic progressions.\par
As for balance, the normal sequence
\[\exseq{\O(2)}{S^*\otimes Q}{N_{C/X}}\]
plus Lemma \ref{match} show that if the slope $s=s(N_{C/X})$ 
satisfies  $[s]=2$ and $S^*\otimes Q$ is unbalanced,
then so is $N_{C/X}$. Explicitly, the slope condition is
\[[\frac{en-2}{k(n-k)-1}]=2.\]	
So whenever this holds and $e$ is not divisible by either $k$ or $n-k$, 
any rational curve of degree $e$ in $X$
is unbalanced. For example, when $n=2k$ the condition on $e$ is
\[k<e<3k/2-1/2k.\]
A general rational curve with degree in this range will be nondegenerate
(i.e. correspond to a nondegenerate scroll in $\P^{n-1}$),
unbalanced and  ambient-unbalanced. 

\end{example}
		Thus, for general Fano manifolds one may expect topological obstructions
		on a curve to be ambient-balanced or balanced, though there remains the possibility
		that all curves of sufficiently high degree are balanced.
		For Fano hypersurfaces of degree $d<n$ in $\P^n$ we will show below that the set of 
		degrees of ambient-balanced or balanced  curves contains some arithmetic progressions,
		resembling the situation for Grassmannians, while for $d=n$ this set 
		contains all sufficiently large integers.\par
		A curve $C\to X$ is said to be ultra ambient-balanced if $T_X|_C$ is ultra-balanced.
		Similarly as in Proposition \ref{ultra-k}, ultra ambient balance has an application to
		intersection numbers. Let
		\[\phi:M_{g, t}(X)\to M_{g, t}\]
		be the natural map and $\kappa=\phi^*(L)$ for some birationally ample $L$. Now define
		\[I_M(u_0, u_1,...,u_t)=\int\limits_M \kappa^{u_0}\eta_1^{u_1}...\eta_t^{u_t}.\]
		\begin{prop}
			Notations as above, assume $C$ is ultra ambient-balanced (resp.
			ambient-balanced) rather
			than ultra-balanced and $t>0$. Let
			\[u_0=\dim(M_{g, t})=3g-3+t
			 \]
			Then for all $u_1,...,u_t$ such that 
			\[\sum u_i=\chi(N)-u_0=(C.-K_X)-n(g-1)-t,\] 
			(resp. and $u_1=...=u_t=n$), we have
			\[I_M(u_0,...,u_t)>0.\]
			\end{prop}
		The proof is similar to that of Proposition \ref{ultra-k}. Note that the case of a general exponent
		vector $(u_0,...,u_t)$ of weight $\chi(N)$ remains open.
		
	\section{Relative and log tangent bundles}\label{relative-sec}
\subsection{Degeneration of tangent bundles}	We construct a relative version of the 
tangent bundle for a family of varieties
	degenerating to normal crossings of multiplicity 2.
	We begin with some local considerations. Consider the surface $X$ with equation
	$x_1x_2=t$ in $\A^3$ with its $t$-projection $\pi:X\to\A^1$. There is
	an associated derivative map
	\[d\pi:T_X\to\pi^*T_{\A^1}\]
	which is clearly surjective except at the node, i.e. the origin, and
	has image $\m\pi^*T_{\A^1}$, where $\m$ is the ideal of the origin.
	Its kernel is invertible and locally generated by the vector field
	\[v=(x_1\del_{x_1}+x_2\del_{x_2})/2+t\del_t.\]
	Now working globally, let
	\[\pi:\mathcal X\to B\]
	be a flat morphism of a smooth variety to a smooth curve whose general fibre is smooth and whose special fibres
	have at most normal crossing double points along a smooth subvariety
	$\Delta$ of codimension 2 (codimension 1 in $\pi\inv(\pi(\Delta))$). 
	Again there is a derivative map
	\[d\pi: T_{\mathcal X}\to\pi^*T_B.\]
	Because $\pi$ can be locally modelled by the above curve fibration, it follows that
	the the image of $d\pi$ is $\I_\Delta\pi^*T_B$ and its kernel, denoted $T_{\mathcal X/B}$
	and called the \emph{relative tangent bundle} of the fibration $\pi$, is locally free. Thus we have
	an exact sequence
	\eqspl{relative}{
		\exseq{T_{\mathcal X/B}}{T_{\mathcal X}}{\I_\Delta \pi^*T_B}.
}	
	In 
	fact $T_{\mathcal X/B}$ is locally near $\Delta$ generated by $v$ as above together with the
	complementary vector fields $\del_{x_3},...$ tangent to $\Delta$. 
	Note that for a smooth fibre $X_t$, we have
	\[T_{\mathcal X/B}|_{X_t}=T_{X_t}.\]
	On the other hand for a singular fibre $X_0$ with normalization $\tilde X_0$ and
	double locus $\Delta\subset \tilde X_0$, the pullback $T_{\mathcal X/B}|_{\tilde X_0}$
	is generated by $x_1\del_{x_1}$ or $x_2\del_{x_2}$ plus the complementary fields. Therefore we have
	\[T_{\mathcal X/B}|_{\tilde X_0}=T_{\tilde X_0}(-\llog{\Delta}).\]
	In particular if $X_0=X_1\cup X_2$ is a union of smooth components then
	\[T_{\mathcal X/B}|_{X_i}=T_{X_i}(-\llog{\Delta}), i=1,2.\]
	Note the exact sequences
	\[\exseq{T_{X_i}(-\Delta)}{T_{X_i}\mlog{\Delta}}{T_\Delta}, i=1,2\]
	which induce
	\eqspl{log-tang}{\exseq{\O_\Delta}{T_{X_i}\mlog{\Delta}|_\Delta}{T_\Delta}}
	where the $\O_\Delta$ subsheaf is locally generated by $x_1\del_{x_1}$ or $x_2\del_{x_2}$.
	The latter sequence is compatible with the identifications 
	\[T_{X_1}\mlog{\Delta}|_\Delta\simeq T_{X_2}\mlog{\Delta}|_\Delta\simeq T_{\mathcal X/B}|_\Delta.\]
	
	\subsection{ Restriction on curves}
	Note that given a smooth pair $(X_i, \Delta)$ and a
	 curve $C_i\subset X_i$ meeting $\Delta$ transversely in $\delta=\Delta\cap C_i$, 
	the restriction $T_{X_i}\mlog{\Delta}|_{C_i}$
	is just the elementary corank-1 down modification of $T_{X_i}|_{C_i}$ at $\delta$ corresponding to the tangent
	hyperplanes $T_p\Delta\subset T_pX_i, p\in\delta$.
	 This has the following immediate consequence
	\begin{cor}\label{relative}
		In the above notations let $\mathcal C/B\to\mathcal X/B$ be a family of curves with special fibre
		$C_0=C_1\cup_{\delta}C_2\subset X_1\cup_\Delta X_2$. 
		Then there is a bundle $T=T_{\mathcal X/B}$ on $\mathcal X$ such that for a general fibre
		$C_t\subset X_t$ we have
		\[T|_{C_t}=T_{X_t}|_{C_t}\]
		while on the special fibre, $T|_{C_i}$ for $i=1,2$
		 is the elementary corank-1 down modification of $T_{X_i}|_{C_i}$
		at the points $p\in\delta$ corresponding to the hyperplanes $T_p\Delta\subset T_pX_i$.
	%	\[T|_{C_i}=T_{X_i}(-\llog{\Delta})|_{C_i}, i=1,2.\]
		\end{cor}
	\begin{example}\label{log-on-line}
		With notations as above, suppose $C_2$ is a $\P^1$ with trivial normal bundle 
		$N_{C_2/X_2}=(n-1)\O$ and $\delta=\{p\}$.
		Then $T_{X_2}|_{C_2}=T_C\oplus (n-1)\O=\O(2)\oplus (n-1)\O$, so that
		 \[T|_{C_2}=T_{X_2}\mlog{\Delta}|_{C_2}=\O(1)\oplus (n-1)\O\]
		where  the
		$(n-1)\O$ quotient coincides at $p$ with the $T_\Delta$ quotient.
		There is an analogous and compatible quotient on the $X_1$ side.
		Then for a point $q\neq p\in C_2$, we can identify $H^0(T|_{C_1\cup C_2}(-q))$
		with the kernel of the natural map 
		\[H^0(T_{X_1}\mlog{\Delta}|_{C_1})\to
		T_{p,\Delta}.\] Therefore 
		\[H^0(T|_{C_1\cup C_2}(-q))=H^0(T_{X_1}|_{C_1}(-p)).\]
		More is true. 
		In fact as in \cite{caudatenormal}, \S 1, there is a modification $T\to T'$
		with cokernel on $C_2$ such that
		\[T'|_{C_2}=n\O\]
		while $T'|_{C_1}$ is the elementary up modification of $T_{X_1}\mlog{\Delta}|_{C_1}$
		at $p$
		corresponding to the $\O_\Delta$ subsheaf as in \eqref{log-tang}, which 
		clearly coincides 	with $T_{X_1}|_{C_1}$ itself, i.e.
		\[T'|_{C_1}=T_{X_1}|_{C_1}.\]
		In particular, given a point modification of $T'|_{C_2}$ leading to an exact sequence
		\[\exseq{K}{T'|_{C_1\cup C_2}}{k\O_q}, q\neq p\in C_2\]
			then there is a corresponding exact sequence
			\[\exseq{K_1}{T_{X_1}|_{C_1}}{k\O_p}\]
			such that
			\[H^0(K)=H^0(K_1).\]
			This argument evidently extends to the case 
			where $C_2$ is a disjoint union of lines with trivial normal bundle.
		The upshot is that such components may effectively be ignored and the log tangent
		bundle $T_{X_1}\mlog{\Delta}|_{C_1}$ replaced by by $T_{X_1}|_{C_1}$ near $C_1\cap C_2$.
		This situation occurs in the proof of Theorem \ref{bmb-antican-thm}
		and Theorem \ref{fano}.
		\end{example}
	\subsection{Log tangents for projective bundle pairs} Let $\pi:X=\P(G)\to B$ 
	be a projective bundle and let $Y=\P(G/A)\subset X$ be
	a codimension-1 projective subbundle, corresponding to a line subbundle $A\subset G$.
	Let $S_G$ be the kernel of the canonical surjection $\pi^*G\to\O_X(1)$.
	Then we have the relative tangent bundle \[T_{X/B}=S^*_G\otimes\O_X(1).\]
	Note that $Y$ is the zero-divisor of the natural map $A\to\O_X(1)$, hence
	\[N_{Y/X}=A^*\otimes\O_Y(1)\]
	where $\O_Y(1)$ is the restriction of $\O_X(1)$.
	Then we have an exact sequence
	\[\exseq{T_{X/B}\mlog{Y}}{S^*_G\otimes\O_X(1)}{A^*\otimes\O_Y(1)}.\]
	Now given a curve $C\to B$, a lifting $C\to X$ corresponds to an invertible quotient $G_C\to M$.
Assume that $A_C\to M$ is injective (i.e. $C\cap Y$ is finite). Then we get an exact sequence
	\eqspl{log-proj}{\exseq{T_{X/B}\mlog{Y}|_C}{S^*_G\otimes M}{A^*\otimes M|_{C\cap Y}}.}
	\subsection{log tangents for blowups}
	Let $\pi:\hat X\to X$ be the blowup of a smooth subvariety $Y$ with normal bundle $N_Y$.
	Let $E=\P(\check N_Y)\subset\hat X$ be the exceptional divisor. Then we
	have an exact diagram
	\eqspl{}{
	\begin{matrix}
		0\to&T_{\hat X}\mlog {E}&\to&\pi^*T_X&\to& \pi^*N_Y&\to 0\\
		&\downarrow&&\parallel&&\downarrow&\\
		0\to&T_{\hat X}&\to&\pi^*T_X&\to&\O_E(1)&\to 0
		\end{matrix}	
	}
	
For example, let $Y$ be a line in $X=\P^2$ so $E=Y, \hat X=X$. If $L\subset X$ is a general line then clearly
\[T_X\mlog{E}|_L=\O(2,0)\]
with upper subbundle 
$\O(2)$ corresponding to $T_L$. If $L_1, L_2$ are distinct lines then the $\O(2)$
subspaces differ at the intersection point $L_1\cap L_2$, hence
\[T_X\mlog{E}|{L_1\cup L_2}=\O(2,2),\]
i.e. a direct sum of line bundles of total degree 2; 
therefore likewise for a general conic $C_2\subset\P^2$.	\par
Now let $Y$ be a line in $X=\P^3$ and $C_2$ a conic in a hyperplane $H\subset X$
containing $Y$, with birational transform $\hat H\subset\hat X$. Then letting $C_2'\subset \hat H$
denote the birational transform of $C_2$, we have $\O_{\hat H}(\hat H)|_{C_2'}=\O_{C'}$, consequently
\[T_{\hat X}\mlog{E}|_{C_2'}=\O(2,2,0)\] 
with upper subsheaf $\O(2,2)$ coming from $T_{\hat H}\mlog{Y}$. Now if $L\subset\hat X$
is the birational transform of a general line meeting $C_2'$ is a point then
$T_{\hat X}|_L=T_{\hat X}\mlog{E}|_L=\O(2,1,1)$.
Therefore as above we get
\[T_{\hat X}\mlog{E}|_{C_2'\cup L}=\O(3,3,2),\]
therefore likewise for $C'_2\cup L$ replaced by
 $C'_3\subset\hat X$, the birational transform of a twisted cubic meeting $Y$
in 2 points.
\par Continuing in the way, we can show that that if $\hat X$ is the blowup of $\P^n$
in a line $Y$ and $C'_n$ is the birational transform of a general rational normal curve
2-secant to $Y$, then
\[T_{\hat X}\mlog{E}|_{C'_n}=2\O(n)\oplus(n-2)\O(n-1).\]
In particular this bundle is balanced.\par
Now an argument similar to but simpler than the one in the proof of Lemma \ref{secant-lem}
below shows that the  balancedness result holds for $Y$ replaced by a linear subspace
of any codimension $c\in[2, n-1]$ as well as $C_n$ replaced by higher-degree rational curves,
so we may conclude:

\begin{lem}\label{log-secant-lem}
	Let $A\subset\P^n$ be a linear subspace of codimension $c\in[2,n-1]$ and let $P\to\P^n$
	be the blowup of $A$ with exceptional divisor $E$. 
	Let $C'\subset P$ be the birational transform of a general rational
	curve $C\subset\P^n$ of given degree $e=n$ or 
	$e\geq 2n-1$ meeting $A$ in $m\leq 2$ points. Then $T_P\mlog{E}|_{C'}$
	is balanced. 
\end{lem}

\section{Curves in projective space}\label{pn-sec}
\subsection{Balanced}
In \cite{alyang}, Atanasov, Larson and Yang construct many semi-balanced curves 
of any genus in projective space. Here we will reprove a subset of result, 
using a method that will be used below for other purposes. The following result is
the method of construction.
%albeit for
%a mire restricted degree range, to construct balanced, rather than semi-balanced curves.
%The refined result will be used below, e.g. in
%the proof of Theorem \ref{main}.
\begin{thm}\label{mainp}
	Let $C_1, C_2\subset\P^n, n\geq 3$, be smooth balanced nondegenerate curves of respective degrees
	$e_1, e_2$, genera $g_1, g_2$, Euler slopes $t_1, t_2>0$ and remainders $r_1, r_2$. 
	Assume  \[ r_1+r_2<n-1.\]
	
	Then\par (i)  there exists a smooth balanced 
	curve $C\subset\P^n$ of degree $e_1+e_2-1$,  genus $g_1+g_2$ and  remainder $r=r_1+r_2$;
	\par (ii) there exists a smooth balanced 
	curve $C'\subset\P^n$ of degree $e_1+e_2-2$,  genus $g_1+g_2+1$ and 
	 remainder $r=r_1+r_2$.
	\end{thm}
\begin{proof} We begin with some numerology.
	Set $g=g_1+g_2, e=e_1+e_2-1$ and
	\[s=\frac{e(n+1)+2g-2}{n-1}, s_i=\frac{e_i(n+1)+2g_i-2}{n-1}, i=1,2.\]
	\[t=[s]+1-g, t_i=[s_i]+1-g_i, i=1,2.\]
	Thus $s=[s]+r/(n-1)$ and likewise for $t, s_i, t_i$.
	Note that $s=s_1+s_2-1$ hence $[s]=[s_1]+[s_2]-1$
	and
	%Hence either (case 1)
	\[ t=t_1+t_2-2\]
%	or (case 2)
%%	\[ t=t_1+t_2-1.\]
%	Also let $r_i$ be the remainder of $s_i$ mod $n-1$. Thus
%	$r_1+r_2<n-1$ (case 1) or $n-1\leq r_1+r_2\leq 2(n-2)$ (case 2).
	\par
	We use the same basic fang construction as in \cite{caudatenormal}. Let
	\[b_1:\mathcal P(\l)=B_{\P^{\l}\times 0}(\P^n_1\times\A^1)\to\P^n_1\times\A^1\]
	be the blow up, which fibres $\pi:\mathcal P(\l)\to\A^1$
	with special fibre $P_0=\pi\inv(0)=P_1\cup_EP_2$ where
	\[P_1=B_{\P^{\l}_1}\P^n_1, P_2=B_{\P^{n-1-\l}_2}\P^n_2,
	E=\P^{\l}_1\times\P^{n-\l-1}_2\]
	and general fibre $\P^n$. $P_0$ is called a fang of type $\l$.\par
	Now for (i), we let $C_i\subset P_i, i=1,2$ 
	be the proper transform of a smooth curve of degree $e_i$
	and genus $g_i$, such that $C_1. E=C_2. E=p$ (transverse intersection)
	and $C_0=C_1\cup_pC_2$. 
	Then the normal bundle $N_{C_i/P_i}, i=1,2$ is an elementary pointwise modification
	of $N_{C_i/\P^n_i}$ of colength $n-1-\l$ (resp $\l$), and 
	under the identification $N_{C_i/P_i}|_p=T_pE$, the kernel
	of the natural map $N_{C_i/P_i}\to N_{C_i/\P^n_i}$ may be identified with
	$T_p\P^{n-1-\l}_2$ (resp $T_p\P^{\l}$). There is an exact sequence
	\eqspl{}{
		\exseq{N_{C_0/P_0}}{N_{C_0/\mathcal P(\l)}}{T^1}	
	}  where $N_{C_0/P_0}, N_{C_0/\mathcal P(\l)}$ are the lci normal bundles,
	$N_{C_0/P_0}=N_{C_1/P_1}\cup_{T_pE}N_{C_2/P_2}$  parametrizes
	compatible deformations of $(C_1, C_2)$ and 
	\[T^1=T^1_{P_0}|_{C_0}=N_{P_0/\mathcal P(\l)}|_{C_0}=T^1_{C_0}\] 
	is a 1-dimensional skyscraper sheaf at $p$.\par
	As the equations defining $C_0$ on $P_0$ restrict to defining equations for each 
	$C_i$ on $P_i$
	\[N_{C_0/P_0}|_{C_i}= N_{C_i/P_i}, i=1,2.\]
	We have exact sequences
	\eqspl{tau}{\exseq{N_{C_i/P_i}}{N_{C_i/\P^n}}{\tau_i}, i=1,2,}
	\eqspl{sigma}{\exseq{N_{C_i/\P^n}(-p)}{N_{C_i/P_i}}{\sigma_i}}
		where $\tau_i$ is a skyscaper sheaf at $p$ of length $\l(\tau_i)=n-1-k, i=1$
		or $k, i=2$, and $\l(\sigma_i)=n-1-\l(\tau_i)$. We have  canonical identifications
		\eqspl{identify}{N_{C_1/P_1}|_p\simeq N_{C_2/P_2}|_p\simeq T_pE.
		} 

	Note that we have subspaces
	\[V_i=N_{C_i/\P^n}(-p)|_p\subset N_{C_i/P_i}|_p, i=1,2\]
	of codimensions $k$ resp $n-1-k$. 
	The image of the restriction map
		\[N_{C_0/P_0}\to N_{C_1/P_1}\oplus N_{C_2/P_2}\]
		and the induced map
	\[H^0(N_{C_0/P_0})\to H^0(N_{C_1/P_1})\oplus H^0(N_{C_2/P_2})\]
	is the inverse image of the 'diagonal' $\Delta$ under the above identification
	\eqref{identify}.
	There is  a standard deformation $\Delta_t$ of $\Delta$ to a $\Delta_0$
	which is union of subspaces, one of them being
 $V_1\times V_2$. 
 This implies firstly that $N_{C_0/P_0}$ admits a specialization to a sheaf that contains
 $N_{C_1/\P^n}(-p)\oplus N_{C_2/\P^n}(-p)$ as cotorsion subsheaf and since that latter
 sheaf has $H^1=0$ (because $t_1, t_2>0$), so does $N_{C_0/P_0}$, i.e.
 \[H^1(N_{C_0/X_0})=0. \]
 It also follows easily that $N_{C_0/X_0}$ is generically generated.\par
 Now the above $H^1$ vanishing implies that,
 possibly after an \'etale base change $A\to\A^1$, $C_0\subset P_0$
 extends to a surface $S$ fibred over $A$. Let $C$ be its general fibre.
 Let $x_{i1},...,x_{it_i-1}, i=1,2$
 be general sections of $S$ specializing to general points of $C_i$.
 Now as $x_{11},...,x_{1t_i-1}, p$ for $i=1,2$ are
 general points on $C_i$ and hence by our hypothesis on $C_1$ and $C_2$,
the restriction map  
\[\rho_0: V_1\times V_2\to
N_{C_0/P_0}|_{\{x_{11},...,x_{1t_1-1}, x_{21},...,x_{2t_2-1}\}}\]
is surjective. 
Therefore the same is true		of $\Delta_t$ for general $t$ hence for $\Delta$ itself
if choose the above identifications generally. Therefore the same is true $N_{C/\P^n}$, which 
shows that $C$ is semi-balanced.\par
For balancedness we argue similarly but, in case $s$ is not an integer,
 add one more section $y$ specializing to a general point on $C_1$.
 Because $C_1$ is balanced, the kernel of the map $\rho_0$ above injects
 into $N_{C_1/\P^n}(-p)|_y$. Therefore the same is true for the kernel
 of the analogous restriction map on $H^0(N_{C_0/P_0})$ therefore ditto for
 $H^0(N_{C/\P^n})$, which proves the injectivity property yielding
 balancedness. This completes the proof of (i).\par
 For (ii), we use the same construction except now $C_i\subset P_i$ meet $E$
 and each other in 2 general
 points $p,q$, so that
 \[C_0=C_1\cup_{\{p,q\}}C_2\]
 has genus $g=g_1+g_2+1$ and 'degree' $e=e_1+e_2-2$.  Note in this case we have
 \[s=s_1+s_2-2, [s]=[s_1]+[s_2]-2, t=t_1+t_2-4.\]
 We have subspaces
 \[V_{ip}=N_{C_i/P_i}(-p-q)\subset N_{C_i/P_i}, i=1,2\]
 and likewise for $q$, and the image of the restriction map
 \[H^0(N_{C_0/P_0})\to H^0(N_{C_1/P_1})\oplus H^0(N_{C_2/P_2})\]
 is the inverse image of the 'bidiagonal' $\Delta_p\times\Delta_q$ under restriction to 
 $\bigoplus\limits_{i=1,2} N_{C_i/P_i}|_{\{p,q\}} $. As above, $\Delta_p\times\Delta_q$
 deforms to $\Delta_{0,p}\times\Delta_{0,q}$ which contains $W:=V_{1,p}\times V_{2,p}
 \times V_{1,q}\times V_{2,q}$.\par
 We consider general sections $x_{ij}, i=1,2, j=1,...,t_i-2$. As above,
 $W$ surjects to $N_{C_0/P_0}|_{x_{11},...,x_{t_2-2}}$ which implies the required surjectivity
 		for $H^0(N_{C_0/P_0})$ and hence for $H^0(N_{C/\P^n})$ for the smoothing $C$, 
 		which proves semi-balancedness. \par
 Now the injectivity statement for balancedness is proven as in part (i).
 This completes the proof.
\end{proof}
\begin{example}\label{balanced-ex}
	
	(i) Taking $e_1=e+2-n, e_2=n, g_1=g_2=0$ in Theorem \ref{main}, (ii) 
	yields ultra-balanced elliptic curves in $\P^n$ of any degree $e\geq 2n-2$.
	In this case $r_2=0, r_1=r$. In particular, the resulting curve is perfect when
	$e=2n-2$.\par
	(ii) Using two ultra-balanced elliptic curves as above  and combining them as in
	Theorem \ref{main}, (i)  yields  a balanced curves
	of genus $2$ and  any degree $e\geq 2(2n-2)-2=4n-6$ in $\P^n$.
	Continuing inductively, we get 
	ultra-balanced curves of genus $g$ and any degree $e\geq g(2n-4)+2$
	in $\P^n$.\par
	(iii) Taking $C_1$ (ultra)- balanced and $C_2$ a rational normal curve (remainder 0) in Part (i)
	yields (ultra) balanced curves of degree $e_1+n-1$ and genus $g_1$. Taking such $C_1, C_2$ in
	Part (ii) yields balanced curves of degree $e_1+n-2$ and genus $g_1+1$.\par
	Continuing inductively, this yields the following special case of the Atanasov-Larson-Yang result
	\cite{alyang}:
	
\end{example}
\begin{cor}\label{Pn}
	For all $g\geq 1, n\geq 3$ and $e\geq n+g(n-2)$, a general curve of genus
	$g$ and degree $e$ in $\P^n$ is balanced.
\end{cor}
\subsection{Ultra-balanced}
Next we refine the result to yield ultra-balanced curves, at the cost of going to higher degree.
\begin{thm}\label{ultra-balanced-thm}
	For all $g\geq 0$ and $e\geq 2(g+1)n$, $n\geq 3$, a general 
	curve of degree $e$ and genus $g$ in $\P^n$ is ultra-balanced .
\end{thm}
\begin{cor}
	For $e, g, n$ as in Theorem \ref{ultra-balanced-thm}, the conclusion 
	of Proposition \ref{ultra-k} holds for $X=\P^n$ and any $t>0$.
	\end{cor}
\begin{proof}[Proof of Theorem]
%	The proof is by induction of the genus. First for $g=0$ the result follows from the fact that
%	a general rational curve of degree $\geq n$ in $\P^n$ is balanced. Next for $g=1$
%	we use a construction as above with $g_1=g_2=0$. Let $N=N_{C_0/P_0}$. 
	We begin with a lemma.
	\begin{lem}\label{secant-lem}
		Let $A\subset\P^n$ be a linear subspace of codimension $c\in[2,n-1]$ and let $P\to\P^n$
		be the blowup of $A$. Let $C'\subset P$ be the birational transform of a general rational
		curve $C\subset\P^n$ of given degree $e=n$ or 
		$e\geq 2n-1$ meeting $A$ in $m\leq 2$ points. Then $C'$
		is balanced in $P$. 
	\end{lem}
	\begin{proof} 
		The case $m=0$, i.e. the assertion that $C$ is balanced in  $\P^n$, originally due to Sacchiero,
		is reproved as Proposition 19 in \cite{caudatenormal} and the case $m=1$ follows easily
		from that as $N_{C'/P}$ is a general modification of $N_{C/\P^n}$. We will focus on the case
		$m=2$ which is harder, as the modifications involved are not general.
		The proof will proceed analogously to the one in loc. cit.
		\par \ul{Case 1:} $e=n$, i.e.
		$C$ is a rational normal curve.\par
		Consider first the case $\dim(A)=1$, i.e. $c=n-1$, where the claim is that
		\[N_{C'/P}=2\O(n+1)\oplus (n-3)\O(n).\]
		First, for $n=3$, $A$ is a 2-secant line
		of the twisted cubic $C$ and $C\cup A$ is a (2,2) complete intersection, so $C'$ is a complete
		intersection of type $(\O(2)-E, \O(2)-E)$ in $P$, $E$ being the 
		exceptional divisor, hence clearly $N_{C'/P}=2\O(4)$ as desired.\par 
		For $n\geq 4$
		we use induction on $n$ using a degenerated curve of the form $C=L\cup_p C_{n-1}$
		where $C_{n-1}$ is a general rational normal curve in a hyperplane $H\subset\P^n$
		and $A$   is a general 2-secant line to $C_{n-1}$ while $p$ is a general point on  $C_{n-1}$
		and $L$ is a general line through $p$. 
		Let $C'_{n-1}\subset H'\subset P$ denote the proper transforms.
		By induction, we have \[N_1:=N_{C'_{n-1}/H'}=2\O(n)\oplus(n-4)\O(n-1),\] hence
		\[N_2:=N_{C'_{n-1}/P}=N_1\oplus \O(n-3)\]
		where $N_1\subset N_2$ is canonical but not the $\O(n-3)$. 
		Moreover, as in loc. cit. we have
		\[N_{C'/P}|_{C'_{n-1}}=N_1\oplus\O(n-2)\]
		and the image of $N_{C'_{n-1}/P}|_p\to N_{C'/P}|_p$ coincides with the image of $N_1$.
		On the other hand we have $N_{C'/P}|_L=\O(2)\oplus(n-2)\O(1)$ and the upper subspace
		coming from the $\O(2)$ is clearly not in the image of $N_{L'/P}\to N_{C'/P}|_L$ at $p$,
		which coincides with the image of $N_{C'_{n-1}}\to N_{C'/P}|_{C_{n-1}}$ at $p$. The upshot
		is that, as in loc. cit. the $\O_{L'}(2)$ must be glued
		at $p$ to an $\O_{C'_{n-1}}(n-2)$ and consequently we have
		\[N_{C'/P}=2\O(n+1)\oplus (n-3)\O(n),\]
		as claimed.\par
		Next consider the case $c+1\leq n\leq 2c-1$ where we must show
		\[N_{C'/P}=(2n-2c)\O(n+1)\oplus (2c-n-1)\O(n).\]
		Again the proof is by induction on $n$ fixing $c$, 
		where the initial case $n=c+1$ is where $A$ is a line which was just
		concluded. Thus assume $n>c+1$ and consider a degenerated curve
		$C=C_{n-1}\cup_pL$ as above.  Arguing as above we get
		\[N_{C'/P}|_{C'_{n-1}}=(2n-2c-2)\O(n)\oplus (2c+1-n)\O(n-1),\]
		\[ N_{C'/P}|_{L'}=\O(2)\oplus(n-2)\O(1)\]
		where the $\O_L(2)$ must glue at $p$ to a general $\O_{C'_{n-1}}(n-1)$ which implies
		$N_{C'/P}$ has the desired form. 	
		
		Finally consider the case where $A$ has codimension $c$ with $n\geq 2c-1$. Then the claim is
		\[N_{C'/P}=(n+1-2c)\O(n+1)\oplus (2c-2)\O(n).\]
		Again we work by induction on $n$ where the initial case $n=2c-1$ is already known, so
		assume $n>2c-1$. Here a similar argument shows
		\[N_{C'/P}|_{C'_{n-1}}=(n-2c)\O(n)\oplus (2c-2)\O(n-1),\]
		\[N_{C'/P}|_{L'}=\O(2)\oplus(n-2)\O(1)\]
		and again the $\O_{L'}(2)$ must glue at $p$ to a $\O_{C'_{n-1}}(n-1)$, so we can conclude as above.
		This finally completes the proof of Case 1. \par
		Note that what
		we have proven is equivalent to: if $C\subset\P^n$ is a rational normal curve
		with normal bundle $N\simeq(n-1)\O(n+2)$,
		$p,q\in C$ are general points, $A$ is a general linear space containing
		the line $\overline{pq}$, and $N'$ is the corresponding '$A$-modification', i.e.
		\[N'=\ker(N\to ((N|_p/T_pA)\oplus (N|_q/T_q)))\subset N,\] 
		then $N'$ is balanced.\par
		\ul{Case 2:} $e\geq 2n-1$.\par
		Notations as above, set
		\[\l=n-1-((n+2)(n-1)-2(c-1))\%(n-1).\]
		Using a fang degeneration as in the first part of the proof,
		take a general $\P^{\l}$ meeting the rational normal curve 
		$C$ in 1 point and let $C_1\subset P_1=B_{\P^{\l}}\P^n$
		be the birational transform of $C$; let $C_2\subset P_2=B_{\P^{n-\l-1}}\P^n$ be the birational transform
		of a general rational curve of degree $e-n+1$, 
		so that $C_1\cap E= C_2\cap E=\{y\}$ is 1 point, where $E$ is the exceptional divisor
		in $P_1$ and $P_2$. Then the appropriate $A$-modification
		of $N_{C_1/P_1}$ at $p,q$  (which is also a suitable modification
		of $N'$ above at $y$) is perfect, while $N_{C_2/P_2}$ is balanced. Then 
		\[C_1\cup_y C_2\subset P_1\cup_E P_2\]
		smooths out to a rational curve of degree $e$ in $\P^n$ whose $A$-modification is balanced.
		This completes the proof of the Lemma.
		%\begin{comment}		
		
		% ******	Let $u_1,...,u_k$ be weights such that
		% 	\[u_1+...+u_k=\chi(N_{C'/P})=\chi(N_{C/\P^n})-\l(c-1).\]
		% 	Pick general points $x_1,...,x_k\in C$ and
		% 	and let $U_1,...,U_k\subset\P^n$ be general linear subspaces of respective codimensions
		% 	$u_1,...,u_k$ through $x_1,...,x_k$ respectively. 
		% 	Let $H$ be a parameter space of balanced, hence ultra-balanced rational curves of degree $e$
		% 	in $\P^n$ and consider the variety of pairs
		% 	\[W=\{(B, L): B\cap U_i\neq\emptyset, 1=1,...,k,  L\cap B= \l\  \mathrm{distinct\  points} \}
		% 	\subset H\times\mathbb G(n-c, n).\]
		% 	Then  
		% 	$C$ together with a general $(n-c)$-plane $A$ containing $\l$ general points on it
		% 	yield a point of $W$ such that the image of $W$ in the space $H$ of curves is smooth,
		% 	as is the fibre over $H$, therefore so is $W$ itself near $(C, A)$.
		% 	Also $W$ dominates $\gG(n-c,n)$ by ultra-balancedness. Consequently, thanks to char. 0,
		% 	$W$ is generically unramified over $\mathbb G(n-c, n)$. 
		% 	But the relative tangent space of $W$ over $\mathbb G(k,n)$ is just 
		% 	$H^0(N^{(u_1,...u_k)}_{C'/P})$. 
	\end{proof}

 Now the proof of  Theorem \ref{ultra-balanced-thm} is by induction on the genus. 
 First for $g=0$ the result follows from the fact that
 a general rational curve of degree $\geq n$ in $\P^n$ is balanced
 which is equivalent to ultra-balanced. Next we
 consider the case $g=1$
 using a fang construction as 
 in the proof of Theorem \ref{mainp}, with $g_1=g_2=0, e_1+e_2=e+2, e_1, e_2\geq 2n$ but with
 \[C_0=C_1\cup C_2, C_1\cap E= C_2\cap E=\{p,q\}\]
 ($e_1, e_2$ and $\l$ are to be determined). 
 By  Lemma \ref{secant-lem}, we may assume each $C_i$ is ultra-balanced in $P_i$. 
 \par
 Let $N=N_{C_0/P_0}$. 
 Let $(u_1,...,u_t)$ be any weight vector with each $u_i\in [1, n-1]$, such that
 \[u_1+...+u_t=\chi:=\chi(N)=e(n+1), e=e_1+e_1-2, e_i=\deg(C_i).\]
 Let $N^u=N^{(u_1,...,u_t)})$. We will show $H^0(N^u)=0$, so that $N$ is $(u_1,...,u_t)$-balanced.
 Set \[N_i=N_{C_0/P_0}|_{C_i}=N_{C_i/P_i},\]\[ 
 \chi_i=\chi(N_{C_i/\P^n})=e_1(n+1)+(n-3),\]\[ \chi'_i=\chi(N_i)=e_1(n+1)+(n-3)-2\l_i,  i=1,2,\]
 where $\l_1=\l-1, \l_2=n-\l$. Then
 \[u_1+...+u_t=\chi'_1+\chi'_2-2(n-1).\]
 \begin{lem} \label{rearrangement} By choosing $e_1, e_2$ properly and relabeling
 $u_1,...,u_t$,  we can arrange things so that
 \eqspl{chi}{u_1+...+u_s=\chi'_1-(n-1), u_{s+1}+...+u_t=\chi'_2-(n-1).}
  \end{lem}
  \begin{proof}[Proof of Lemma] It suffices to arrange that
 \[2\l_1=\chi_1-(n-1)-(u_1+...+u_s)=e_1(n+1)-2-(u_1+...+u_s)\]
 for then the other equality in \eqref{chi} is automatic.
 Let $u_1+...+u_s$ be a maximal sub-sum that is $\leq\chi_1-(n-1)=e_1(n+1)-2$. Then
 \[\chi_1-2(n-1)\leq u_1+...+u_s\leq\chi_1-(n-1),\] 
  \[\chi_2-2(n-1)\leq u_{s+1}+...+u_t\leq\chi_2-(n-1).\] 
 If either  $d_1:=\chi_1-(n-1)-(u_1+...+u_s)$ or the analogous $d_2$ is even we can just set 
 \[\l_i=(\chi_1-(n-1)-(u_1+...+u_s))/2\]
 and \eqref{chi} holds. Hence we may assume $d_1$ and $d_2$ are odd. Assume first that $n$ is odd, hence
 we may also assume $u_s$ is odd. If 
 \eqspl{u}{u_1+...+u_{s-1}\geq\chi_1-2(n-1)} we may just replace
 $s$ by $s'=s-1$ and be done. If \eqref{u} fails we may replace $e_1$ by $e_1'=e_1-1$ and $e_2$ 
 by $e'_2=e_2+1$ and then be done. \par Now Assume $n$ is even. If 
 \[u_1+...+u_s\equiv e_1\mod 2\] we can just
 let \[\l_1=(e_1(n+1)-2-(u_1+...+u_s))/2.\] Otherwise, we just let $e'_1=e_1+1$ and 
 $e'_2=e_2-1$ and work with $e'_1, e'_2$ instead. QED claim.
 \end{proof}
 Now as $N_1^u:=N_1^{(u_1,...,u_s)}, N_2^u:=N_2^{(u_{s+1},...,u_t)}$ 
 are balanced and have $\chi=\rk= n-1$, we have
 \[N_1^{u}=(n-1)\O_{C_1}, N_2^{u}=(n-1)\O_{C_2}.\]
 Now let $E_i\subset P_i, i=1,2$ be the exceptional divisor (a copy of $E$). 
 Then $P_0$ is constructed using
 an arbitrary isomorphism $\phi:E_1\to E_2$ and I clam that by choosing $\phi$ sufficiently general,
 we can ensure that
 \[H^0(N^{u})=H^0(N_1^{u}\cup_{p,q}N_2^{u})=0,\]
 i.e. no nonzero sections of $N_1^u$ and $N_2^u$ agree on $p$ and $q$. 
 Now we have natural isomorphisms
 \[T_pE_i\simeq N_i^u|_p\simeq H^0(N_i^u)\simeq N^u_i|_q\simeq T_q E_i, i=1,2.\]
 It will suffice to choose the isomorphism $\phi$, which may be identified
 as an arbitrary automorphism of $\P^\l\times\P^{n-1-\l}$, so that 
 the derivative $d_p\phi-d_q\phi$ is nonsingular where $d_p\phi:T_pE_1\to T_pE_2$
 is the derivative and likewise for $q$. By suitable identifications, we may assume $d_p\phi$
 is the identity $I$ while $d_q\phi$ is an arbitrary trace-0 matrix $M$. Then clearly for
 suitable $M$ (e.g. non-scalar diagonal), $M-I$ is nonsingular. This completes the proof for genus 1.
 \par
 Now for $g>1$ we argue by induction, using a fang degeneration as above but with
 \[C_0=C_1\cup_pC_2\subset P_0=P_1\cup P_2, g_1=1, g_2=g-1.\]
 Using notations as above, we let $u=(u_1,...,u_t)$ be any weight vector
 with $\chi(N^u)=0$.
 We may assume $C_1, C_2$ are ultra-balanced and that $p$ is general on $C_1, C_2$.
 An argument similar to the proof of Lemma \ref{rearrangement} above but simpler shows that
we may assume by choosing the fang type (i.e. $\l$) suitably
that \[\chi(N_1^{(u_1,...,u_s)})=0, \chi(N_2^{(u_{s+1},...,u_t)})=n-1.\]
By ultra-balancedness we have first $H^0(N_1^u))=0$, then because $\chi(N_2^u(-p))=0$,
also $H^0(N_2^u(-p))=0$. Hence $H^0(N^u)=0$.

	\end{proof}
\subsection{Ambient-balanced}
The analogue of Theorem \ref{mainp} for ambient-balanced curves also holds:
\begin{thm}\label{ambient}
	Let $C_1, C_2$ be as in Theorem \ref{mainp} and assume moreover\par (i) $C_1, C_2$ are ambient-balanced;\par
	(ii) the ambient remainders $r_1=e_1\%n, r_2= e_2\%n$ satisfy $r_1+r_2<n$ (e.g. $n|e_1$).\par  Then \par
	(i)  there exists a smooth ambient-balanced
	curve $C\subset\P^n$ of degree $e_1+e_2-1$,  genus $g_1+g_2$ and ambient remainder $r=r_1+r_2$;
	\par (ii) there exists a smooth ambient-balanced
	curve $C'\subset\P^n$ of degree $e_1+e_2-2$,  genus $g_1+g_2+1$ and ambient remainder $r=r_1+r_2$.
	\end{thm}
\begin{proof}
	We follow the general outline of the proof 
	of Theorem \ref{mainp} but now taking $C_1$ and $ C_2$ in the same $\P^n$.
	By assumption $t(N_{C_i/\P^n})\geq 2, i=1,2$ so we may assume $C_1\cap C_2$ is exactly 1 general
	point (Case (i)) or 2 general points (Case (ii)). 
	Then as in the above proof it follows
	that $C_1\cup C_2$ is smoothable in $\P^n$. From Lemma \ref{reducible} it follows that $T_{\P^n}|_{C_1\cup C_2}$
	is semi-balanced, hence this is true for the smoothing as well.
	\end{proof}
\begin{cor}\label{ambient-cor}
	For all $g\geq 0, n\geq 4$ and $e\geq n+g(n-2)$, there exists a balanced and ambient-balanced,
	hence moduli-interpolating curve of genus
	$g$ and degree $e$ in $\P^n$.
	\end{cor}
\begin{proof} The case $g=0$ is well known (balancedness by Sacchiero \cite{sacchiero}, ambient-balancedness e.g. by Lemma
	26 of \cite{caudatenormal}), so assume $g\geq 1$.
	By Corollary \ref{Pn} there exists such a curve $C'$ that is balanced. 
	Using Theorem \ref{ambient} with $C_1$ a rational normal curve, 
	it follows similarly using induction on $g$ that there is such a curve $C"$
	that is ambient-balanced. Because $C', C"$ are non-special, the family of curves of
	degree $e$ and genus $g$ in $\P^n$ is irreducible, hence the general curve  $C$
	in the family is balanced and ambient-balanced.
	\end{proof}
Finally, we will prove an analogue of Theorem \ref{ultra-balanced-thm} for ambient balanced curves.
\begin{thm}
	For $e\geq 3g+n+1, n\geq 3$, there exists an  
	ultra ambient-balanced curve of degree $e$ and genus $g$
	in $\P^n$,
	\end{thm}
Using Theorem \ref{ultra-balanced-thm}, we conclude
\begin{cor}\label{ultra-bal-Pn}
	For $e\geq 2(g+1)n, n\geq 3$, a general curve of degree $e$ and genus $g$ in $\P^n$
	is ultra-balanced and ultra-ambient balanced.
	\end{cor}
\begin{proof}[Proof of Theorem]
	The proof is analogous to that of Theorem \ref{ultra-balanced-thm} and proceeds by induction on
	 the genus. The case $g=0$ follows from the fact that balanced = ultra balanced in genus 0. 
\par	 We next take up the case 
	  $g=1, e\geq 3, n\geq 2$.
	  beginning with $n=2, e=3$.
 		In this case what must be shown is that for a weight-vector
 		\[u=(u_1,...,u_t), u_i\in\{1,2\}, |u|=9,\] 
 		and for a general cubic $C$, we have 
 		\[H^0(T_{\P^2}^u|_{C})=0.\] As this is an open property of $C$
 		we may consider a reducible
 		cubic  $C=C_2\cup_{p,q} L$ with $C_2$ a conic and $L$ a line.
 		Then we have
 		\[T_{\P^2}|_L=\O(2,1), T_{\P^2}|_{C_2}=\O(3,3).\]
Now the weight vector $u$ must have an odd number of components equal to 1, 
with the rest equal to 2, hence
		we may assume $u=(u', u")$ with $|u'|=5$ and then $H^0(T_{\P^2}^{u'}|_L)=0$
 		and $H^0(T_{\P^2}^{u"}|_{C_2}(-p-q)=0$. Consequently
 		$H^0(T_{\P^2}^u|_C)=0$, which proves the result for cubics in $\P^2$.\par
 		Next we will prove by induction on $ n\geq 2$ 
 		that for a general cubic $C$ in a plane in $\P^n$, 
 		$C$ is ultra ambient balanced in $\P^n$.
 		The proof is by induction on $n$ with $n=2$ already known so assume $n\geq 3$ and note that
 		\[T_{\P^n}|_C=T_{\P^{n-1}}|_C\oplus L, L:=\O(1)|_C.\]
 			If $u=(u_1,...,u_t), u_i>0$ is a weight vector of weight $|u|=3(n+1)$
 			then $t\geq 3$ so we can write $u=u'+u", u"=(1,1,1,0,...,0)$
 			and then 
 			\[H^0(T_{\P^n}|_C^u)=H^0(T_{\P^{n-1}}|_C^{u'})\oplus H^0(L^{u"}).\]
 			Now the first summand vanishes by induction and the second by inspection.
 			Thus $T_{\P^n}|_C$ is ultra-balanced.\par
 		
 		Next we conside the case of a general degree  $e\geq 3$ and $g=1$, 
 		working by induction on $e$.  Consider a curve of the form
 		$C^1_{e+1}=C^1_e\cup_pL$ where $C^1_e$ is elliptic
 		and $L$ is a 1-secant line, and pick a weight vector $u=(u_1,...,u_t)$
 		with $|u|=\chi(T_{\P^n}|_{C^1_{e+1}})=(n+1)(e+1)$. Note that
 		\[T_{\P^n}|_L=\O(2)\oplus(n-1)\O(1).\]
 		Write $u=(u',u")$ with $|u'|$ maximal subject to $|u'|\leq\chi(T_{\P^n}|_L)=2n+1$,
 		so that $|u'|\geq n+1$ and also 
 		\[(n+1)e-n\leq u"\leq (n+1)e=\chi(T_{\P^n}|_{C^1_e}).\]
 		Write $u'=(u_1,...,u_s)$ and let the quotients $U_1,...,U_s$ be supported on $L$.
 		Then the restriction maps
 			\[\rho_1: H^0(T_{\P^n}|_L^{u'})\to T_{\P^n}|_p, \rho_2: H^0(T_{\P^n}|_{C^1_e}^{u"})\to T_{\P^n}|_p\]
 			are injective by inspection (resp. induction).
 		Considering $N_L(-1)$ trivialized, the quotients $U_1,...,U_s$ are general mod $T_pL$
 		while $T_pL$ itself may be chosen generally fixing $C^1_n$. Therefore the images
 		of $\rho_1, \rho_2$
 		%the (inductively, injective) maps
 		%\[H^0(T_{\P^n}|_L^{u'})\to T_{\P^n}|_p, H^0(T_{\P^n}|_{C^1_e}^{u"})\to T_{\P^n}|_p\]
 		are in general position, i.e. complementary. Therefore $H^0(T_{\P^n}|_{C^1_{e+1}}^u)=0$. 
 		This finally proves the Theorem for $g=1$.
 		\par
 		Now for $g>1$ we argue by induction on $g$ and 
 		can just copy over the last part of the proof of Theorem
 		\ref{ultra-balanced-thm},  using a fang curve
 		\[C_1\cup_pC_2\subset P_1\cup_EP_2\]
 		with $C_1$ elliptic and $C_2$ of genus $g-1$,
 		and using the relative tangent bundle $T_{\mathcal P(\l)/\A^1}$
 		discussed in \S \ref{relative-sec} instead of the relative normal bundle . 
 		$C_1$ and $C_2$ may be assumed ultra ambient-balanced in $\P^n$
 		and consequently $T_{P_i}\mlog{E}|_{C_i}, i=1,2$ is ultra-balanced as well. 
 		Appropriately distributing weights and degrees among $C_1, C_2$ as in the latter proof, 
 		it goes through essentially verbatim.

	\end{proof}
	\section{Curves in anticanonical hypersurfaces}\label{antican-sec}
	The purpose of this section is to prove our results constructing (ultra) balanced 
	and ambient-balanced curves
	on anticanonical hypersurfaces. The construction is based on the following result:
	\begin{thm}\label{main}
	Suppose the exists a balanced (resp. ultra-balanced, resp. semi-balanced) 
	curve of degree $e_1$ and genus $g$ in $\P^{n-1}, n\geq 4$.
	Then for all $e$ with $(n-1)(e_1-1)\leq e\leq (n-1)e_1$ (resp. for $e=(n-1)e_1$), 
	there exists a balanced (resp. ultra-balanced, resp. semi-balanced) 
	curve of genus $g$ and degree $e$ on
	a general hypersurface of degree $n$ in $\P^n$.
	\end{thm}
%/************	
%Combining this with Example \ref{balanced-ex}, we conclude
%	\begin{cor}\label{interpol-cor}
%	A general hypersurface of degree $n$ in $\P^n$ satisfies the interpolation property for curves
%	of any genus $g$ and degree $e\geq (n-1)( n+ g(n-2))$.
%	\end{cor}
%Combining with the result of Atanasov-Larson-Yang \cite{alyang}, we conclude
%\begin{cor}\label{weak-interpol}
%	A general hypersurface of degree $n$ in $\P^n$ satisfies the 
%	semi-interpolation property for curves
%	of any genus $g$ and degree $e=e_1(n-1), e_1\geq g+n-1$.
%\end{cor}
%As explained in \cite{caudatenormal}, \S 1, interpolation can be considered from
%the viewpoint of a general point-set on $X$ rather than a good curve, and
%converting Corollary \ref{interpol-cor} to that viewpoint then yields the following result
%(where the regularity hypothesis is included to rule out oversize families).
%
%\begin{cor}
%	For a general hypersurface $X$ of degree $n\geq 4$ in $\P^n$, any genus $g$ 
%	and any integer $q\geq 2+\frac{(2n^2-3n+8)g-2}{n-2}$, a general collection
%	of $q$ points on $X$ lies on a regular curve of genus $g$ and degree
%	$e=(n-2)q+(n-4)(g-1)$ and none lower.
%	\end{cor}
%\begin{rem}
%	We do not exclude the possibility the points may lie on a lower-degree curve of genus $g$
%	that is irregular (member of an oversize family).
%	\end{rem}
%*****************/
\begin{proof}%[Proof of Theorem]
	We begin with the balanced and ultra-balanced cases.
	For $g=0$ this is contained in Theorem 20 in \cite{caudatenormal}, and the
	proof for general $g$ proceeds along similar lines, modulo the constructions of the last section
	for higher-genus curves in $\P^n$.
\par	Assume to begin with that $C\subset\P^{n-1}$ is balanced 
(resp. ultra-balanced)  of degree $e_1$
	and genus $g_1$ as in Corollary \ref{ambient-cor}.
Write \[e=e_1(n-1)-a, 0\leq a\leq n-1.\]
We start with the same setup as in the proof of Theorem \ref{mainp}. Thus consider a fan 
\[\mathcal P=B_b(\P^n\times \A^1)\to \A^1\] with special fibre 
\[P_0=P_1\cup_E P_2, P_1=B_b\P^n, P_2=\P^n, E=\P^{n-1}.\] Now in $\mathcal P$ we consider a general
relative hypersurface $\mathcal X$ of type $(n, n-1)$ with special fibre 
\[X_0=X_1\cup_F X_2\]
where: $X_1$ is the blow up at $b\in\P^n$ of a general hypersurface in 
$\P^n$ of degree $n$ and multiplicity $n-1$ at $b$,
 with exceptional divisor $F$;
 and $X_2$ is a general hypersurface of degree $n-1$ in $\P^n$
 with hyperplane section $F$. Then, via projection
from $b$, $X_1$ is realized as $\P^{n-1}$ blown up at a general $(n, n-1)$ complete
intersection \[Y=F_{n-1}\cap F_n\]
where the exceptional divisor $F$ becomes the birational transform of $F_{n-1}$. 
\par Now by the discussion in Case 1 of the proof of Theorem 20 
of \cite{caudatenormal}, which uses nothing
about the genus of $C$, we may assume
$Y$ meets  $C$ 
transversely in $a$ general points $p_1,...,p_a$ and its tangents
$T_{p_i}Y$ yields general hyperplanes in the normal space $N_{C_1}(p_i), i=1,...,a$.
If $C_1, F$ denotes the birational transform of $C_1$
resp. $F_{n-1}$ in $X_1$, then
$N_{C_1/X_1}$ is a general down modification of $N_{C_1/\P^{n-1}}$
at $p_1,...,p_a$, hence it is balanced by Lemma \ref{down} 
(resp. ultra-balanced by definition). Then set
\[\{q_1,...,q_e\}=C\cap F_{n-1}\setminus\{p_1,...,p_a\}=C_1\cap F\] and
\[C_0=C_1\cup (\bigcup\limits_{i=1}^e L_i)\]
where $L_i$ is a general line in $X_2$ through $q_i$.
Because $N_{L_i/X_2}$ is a trivial bundle, it is easy to check that
$N_{C_0/X_0}$ is balanced (resp. ultra-balanced) around $C'_1$. Therefore when 
$(C_0, X_0)$ smooth out to a general $(C, X)$, $X$ a general hypersurface of degree $n$,
the normal bundle $N_{C/X}$ is likewise balanced (resp. ultra-balanced). This proves the assertion
of the Theorem in the balanced and ultra-balanced cases.\par
Note that in the above argument, if $C_1$ is semi-balanced and $a=0$, then $C_0$
is semi-balanced around $C_1'$ hence its smoothing $C$ is semi-balanced.
This proves the assertion in the semi-balanced case.
\end{proof}
Now Theorem \ref{ultra-balanced-thm} yields:
\begin{cor}\label{ultra-anticanonical-cor}
	For $n\geq 4$ a general hypersurface of degree $n$ in $\P^n$ contains ultra-balanced curves of
	genus $g$ and degree $e$ for all $e\geq 2(g+1)n(n-1)$.
	\end{cor}
\begin{rem}
	Trying to prove even semi-balancedness for $C_0$ when $e$ is not a multiple of $n-1$
	requires modifications of the normal bundle to $C_1$ and hence an assumption
	that $C_1$ be balanced, rather than weakly balanced.
	\end{rem}
A modification of this approach yields curves that are both balanced and ambient-balanced:
\begin{thm}\label{bmb-antican-thm}
	A general hypersurface of degree $n$ in $\P^n$, $n\geq 4$, contains ultra-balanced and
	 ultra ambient-balanced curves of degree $e$ and genus 
	$g$ provided $g=0, e\geq n-1$ or $g\geq 1, e\geq 4g(n-1)$.
	\end{thm}
\begin{proof} We use the construction and notations in the proof of Theorem \ref{main}.
	Given Corollary \ref{ultra-bal-Pn}, proving Theorem \ref{bmb-antican-thm}
	 is a matter of showing that the curves constructed in the latter proof 
	may be assumed ultra ambient-balanced provided $C\subset\P^{n-1}$ is. 
	We use the relative tangent bundle as 
	discussed in \S \ref{relative-sec}, so the restricted tangent bundle $T_X|_C$
	for a curve  on $ X$ specializes to
	\[T_{X_1}\mlog{E}|_{C_1}\cup T_{X_2}\mlog{E}|_{C_2}, C_1\cup C_1\subset X_1\cup X_2,\]
	where $C_2\subset X_2$ is a disjoint union of lines with trivial normal bundle.
	Now working as in Example \ref{log-on-line}, we modify the relative tangent bundle
	along $C_2$ so the specialized bundle becomes $T_{X_1}|_{C_1}\cup (n-1)\O_{C_2}$.
	Then it is clearly sufficient to show that $C_1\subset X_1$ is ultra ambient-balanced.
	But, with the above notations,  $T_{X_1}|_{C_1}$ is a general corank-1 down modification of 
	of the ultra-balanced bundle $T_{\P^{n-1}}|_C$
	at $p_1,...,p_a$,  hence is ultra-balanced.

	\end{proof}
%\begin{cor}\label{moduli-antican-cor}
%	Notations as above, $X$ contains a separably interploating and
%	 moduli-interpolating curve of genus $g$ and degree $e$
%	provided $e\geq 4g(n-1), g\geq 1$ or $e\geq n-1, g=0$.
%	\end{cor}

%\begin{rem}\label{d<n} Trying to prove similar results in the case
%	$d<n$, i.e. extending the method of \cite{caudatenormal}, \S 6 to higher genus,
%	 would involve extending Lemmas 25, 26
%	of \cite{caudatenormal} to the case of higher genus. For Lemma 25 this is
%	accomplished by Lemma \ref{match} above but for Lemma 26 this is not clear.
%	
%	\end{rem}
\section{Curves in other Fano hypersurfaces} \label{fano-sec}
We now turn our attention to lower-degree hypersurfaces.
The purpose of this section is to prove the following
\begin{thm}\label{fano}
Let $X$ be a general hypersurface of degree $d\in[3,n-1]$ in $\P^n, n\geq 4$. Then\par
(i) $X$ contains balanced 
	  curves $C$ of degree $e$
	and genus $g$ provided there exists\nl $e_0\in [(g+1)n,e]$ such that either
	\eqspl{match-thm}{
		[\frac{-de_0+e}{n-d}]+e=e_0+[\frac{2e_0+2g-2}{d-2}].
		}
	or
	\eqspl{match1-thm}
	{\frac{-de_0+e}{n-d} +e =e_0 +\bigg\lfloor\frac{2e_0+2g-2}{d-2}\bigg\rfloor+1.}
	
	In particular given $g\geq 0$, there exist such balanced curves for every 
	$e$ in at least $(d-2)(n-d+1)$ many arithmetic progressions with difference $d(n-2)$.
	\par (ii) 
	$X$ contains ambient-balanced 
	curves $C$ of degree $e$
	and genus $g$ provided there exists $e_0\in [(g+1)n,e]$ such that
	\eqspl{match-ambient}{	
[\frac{-de_0+e}{n-d}]+e=e_0+[\frac{e_0}{d-1}]
	}
or
	\eqspl{match1-ambient}{	
	\frac{-de_0+e}{n-d}+e=e_0+[\frac{e_0}{d-1}]+1
}
In particular, given $g\geq 0$,
there exist such ambient-balanced curves curves for infinitely many $e$.

	\end{thm}
For the 'in particular' portion of (i) see the Appendix by M. C. Chang below.
\begin{rem}
	(i) Note that for $d>n/2$, eq. \eqref{match-thm} already implies $e>e_0$.
\par (ii) In light of Example \ref{ambient-example}, it is not unreasonable to expect some
obstructions in terms of  $e$ to the existence of an ambient-balanced curve of degree $e$.
\end{rem}
\begin{example}
	Solving \eqref{match-ambient} is elementary. Write
	\[e_0=\alpha(d-1)+\beta, 0\leq\beta<d-1, \alpha=[\frac{e_0}{d-1}],\]
	\[e-de_0=q(n-d)+r, 0\leq r<n-d.\]
	Then an elementary calculation yields
	\[d(d-2)\alpha+(d-1)\beta=(-q)(n-d+1)-r.\]
	This is solvable for $e$ iff
	\[(d-1)e_0-[\frac{e_0}{d-1}]\not\equiv 1\mod n-d+1.\]
Explicitly, writing 
	\[(d-1)e_0-[\frac{e_0}{d-1}]=u(n-d+1)+v, -(n-d)<v\leq 0,\]
	the solution is
	\[e=de_0-u-v.\]
	Because $u\leq{((d-1)e_0+n-d)}/{2}$, clearly $e\to\infty$ as $e_0\to\infty$ so there are
	infinitely many $e$ for given $n, d, g$.
	\end{example}
\begin{example}\label{d=n-1-example} (M. C. Chang)
	For $d=n-1$, equation \eqref{match-thm} reads
	\[2e=ne_0+[\frac{2e_0+2g-2}{n-3}].\]
	Write \[g=x(d-2)+y, e_0=(2k+r)(d-2)+c, 0\leq y,c\leq d-3, r\in\{0,1\}.\]
	Then, setting $t=[(2c+2y-2)/(d-2)]$, we get
	\[e=kd(d-1)+x+(t+r(d^2-d)+c(d+1))/2.\]
	$e$ is an integer iff $t+c(d+1)$ is even. 
	Assuming $c>0$, we have $t\in [0,3]$.
	We try to count the 'bad' pairs $(c, r)\in [1, d-3]\times [0,1]$, i.e. those
	where $t+c(d+1)$ is odd, with $y$ given. If $d$ is odd badness means $t$ is odd, i.e. $t\in\{1, 3\}$.
	The number of such $c$ is at most $d/2-1$. If $d$ is even badness means either
	$t\in\{1,3\}$, $c$ even (at most $((d/2)-1)/2$ solutions)
	or $t\in\{0,2\}$, $c$ odd (again at most $((d/2)-1)/2$ solutions).
	Thus if $d$ is either even or odd, there are at most $d/2-1$ bad $c$ values,
	hence the number of good pairs $(c, r)$ is at least
	$2(d-3-(d/2-1))=d-4$; i.e. there are at least $d-4$ good congruence classes of $e_0$ mod
	$2(d-2)$ hence at least $d-4$ distinct arithmetical progressions for $e$ with difference
	$d(d-1)$.\par
	Similarly treating eq. \eqref{match-ambient} for $d=n-1$ yields
	\[e=(ne_0+[\frac{e_0}{n-2}])/2.\]
	When $n$ is even (resp. odd), this is an integer provided  $[\frac{e_0}{n-2}]$ is even
	 (resp. the  remainder $e_0\% (n-2)$ is even). 
	 This leads to about $n-2$ (resp. $(n-3)/2$)
	arithmetic progressions of $e$ values with difference $n(n-2)$ (resp. $(n-1)^2/2$)
	for $n$ even (resp. odd).
	Note that the condition for \eqref{match-ambient} to hold is,
	in the above notations $2k+r\equiv c\mod d-1$. This yields about $d-4$ arithmetic progressions for $e$ with
difference $d(d-1)^2$.

%	
%	**********
%	This yields $2(d-2)$ distinct arithmetic progressions with difference $d(n-2)$.
%	
%	
%**********	
%If $n$ is even this can be solved for $e$ whenever the remainder $\rho=e_0+g-1\% (n-3)<(n-3)/2$, yielding
%	at least $(n-3)/2$ distinct arithemtic progressions of $e$ values with difference $(n^2-3n+2)/2$. 
%	For \eqref{match-ambient} to hold,
%	$e_0$ should be of the form $k(n-2)$ where $k+g-1\equiv 0\mod n-3$.\par
%	If $n=2n_0+1$ is odd, writing
%	\[e_0=\lambda(n_0-1)+1-g+\rho, 0\leq\rho<n_0-1,\]
%	this is solvable for $e$ whenever
%	\[\rho+\lambda n_0+1-g\equiv 0\mod 2\]
%	so again we get an arithmetic progression of $e$ values.
%	*************
	\end{example}
\begin{example}
	When are the curves produced by Theorem \ref{fano} actually \emph{perfect} ?
	For perfect balance, it is a matter of replacing
 \eqref{match-thm}  by the 'exact' equation
	\eqspl{match-exact}{	\frac{-de_0+e}{n-d}+e=e_0+\frac{2e_0+2g-2}{d-2}}
	together with the condition that both sides of \eqref{match-exact} be integers.
	This is a sufficient condition that the curve $C$ is {perfectly} balanced.
	Assume first that $d$ is odd and write 
	\eqspl{e0}{e_0=\lambda(d-2)+1-g, \lambda\in \Z.} 
	Then the condition that \eqref{match-exact} can be solved for
	an integer  $e$ is
	\[\lambda d(n-2)+n(1-g)\equiv 0\mod n-d+1\]
	or equivalently
	\eqspl{lambda}{\lambda(n+1)(n-2)+n(1-g)\equiv 0\mod n-d+1.}
	At the upper end $d=n-1$, $n$ even, \eqref{lambda} is automatic, 
	so the curves produced by Theorem \ref{fano} are always perfectly balanced. 
	A the lower end, 
	if $d=3$, eq. \eqref{lambda} becomes the condition $2-2g\equiv 0\mod n-2$.
	For $d>3$ odd, \eqref{lambda} admits an arithmetic progression of solutions $\lambda$ 
	(hence of $e$ values yielding perfectly balanced curves) provided
	\[(d, n+1)=1=(d-3, n-2)\]
	For example when  $d=5$ this holds whenever
	$n$ is odd and $n\not\equiv 4\mod 5$. 
%	Again these are sufficient 
%	conditions for the stronger form \eqref{match-exact} with $e_0$ as in \eqref{e0},
%	hence for the existence of a {perfectly} balanced curve of degree $e$ and genus $g$ .
	\par
	Similarly analyzing the case $d=2d_0$ even leads to
	\[(d^2/2-2d+1)\lambda+(d-1)(1-g)\equiv 0\mod n-d+1\]
	which admits an arithmetic progression of solutions $\lambda$ provided
	\[(d^2/2-2d+1, n-d+1)=1.\]
%	\[(2d_0^2-4d_0+1)\lambda+(2d_0-1)(1-g)\equiv 0\mod n-2d_0+1.\]
	Similarly treating eq. \eqref{match-ambient}, i.e. seeking $C$ that is
	perfectly ambient-balanced,  leads to
	\[e=\frac{(n-1)d}{(d-1)(n-d+1)}e_0.\]
	This is solvable at least when $(d-1)(n-d+1)|e_0$, leading to at least one
	arithmetic progression of degrees for which there exists a perfectly ambient-balanced
	curve.
	\end{example}
\begin{proof}[Proof of Theorem]
	The proof proceeds along similar lines as that of Theorem 31 of \cite{caudatenormal},
	using a relative fang. Thus let $\mathcal Z\to\A^1$ be a relative fang of type $(n,m), m=d-1\geq 2$,
	with special fibre
	\[Z_0=Z_1\cup Z_2, Z_1=\P_{\P^m}(1, 0^{n-m}), Z_2=\P_{\P^{n-m-1}}(1, 0^{m-1}).\]
	Let $\mathcal X\subset\mathcal Z$ be a general member of the linear system 
	$|dH-(d-1)Z_2|$ where $H\subset\P^n$ is a hyperplane. The $\mathcal X\to\A^1$
	has special fibre \[X_0=X_1\cup_E X_2.\] Here
	$X_1=\P_{\P^m}(G)$ where $G$ is a bundle on $\P^m$ that fits in an exact sequence
	\eqspl{g-seq}{\exseq{\O(-m)}{\O(1)\oplus (n-m)\O}{G}}
	in which the left map is general. Also  $X_2$ fibres over $\P^{n-m-1}$ with general
	fibre a general hypersurface of degree $d-1=m$ in $\P^{m+1}$. As in the above-referenced proof, 
	we will construct a balanced curve in $X_0$ of the form $C_1\cup C_2$ where 
	$C_1\subset X_1$ is balanced and $C_2\subset X_2$ is a disjoint union of lines in fibres
	of $X_2\to\P^{n-m-1}$ and as such has trivial hence balanced normal bundle.
	Then $X_0$ will smooth along with $Z_0$ to a balanced curve in the general fibre of $\mathcal X\to\A^1$.
	It will suffice to construct $C_1$.\par
	To this end, proceeding as in \cite{caudatenormal}, proof of Theorem 31, 
	we will start with a balanced curve $C_0\subset\P^m$ of genus $g$ and degree $e_0$ and lift it
	to $C_0\simeq C_1\subset \P(G)=X_1$ using a general surjection
	\eqspl{}{
	\psi:G_{C_0}\to M	
	} where $M=\O_{C_0}(H+A)$ with $L=\O(H)$ being the hyperplane bundle from $\P^m$ and
and $A$ is a general effective
divisor $A$ of degree $e-e_0$, $e_0=\deg(L)$, which also coincides with
$C_1.E$. Such a map $C_1\to X_1$ comes from
a map $\phi:C\to\P^n$ corresponding to $n+1$ sections of $L$ among which $m+1$ vanish on $A$,
and can be constructed by starting from $C_0\to\P^m$ corresponding to $m+1$
sections of $L$ and adding $n-m$ additional sections
of $M=L(A)$.\par
 Now setting $K=\ker(\psi)$,
the vertical part of the normal bundle $N_{C_1/\P(G)}$ is $K^*(M)$, i.e.
we have an exact normal sequence
\eqspl{normalseq}{\exseq{K^*(M)}{N_{C_1/\P(G)}}{N_{C_0/\P^m}}}
	and the relation \eqref{match-thm} means exactly that the slope matching condition
	of Lemma \ref{match} and \cite{caudatenormal}, eq. (10) holds. 
Thus will suffice to prove as in \cite{caudatenormal} that $K^*(M)$ is balanced. For $g=0$
this is proved in  \cite{caudatenormal}, Lemma 33. In the general case we will
use induction on $g$, starting with a reducible form of $C_0$ of the form
\eqspl{c00}{C_{00}=C_{01}\cup_{p,q}C_{02}\subset \P^m}
where $C_{01}$ is a rational normal curve (of degree $m$,
$C_{02}$ is a balanced curve of genus $g-1$ and degree $e_{02}\geq m+(g-1)(m-2)$ (see
Corollary  \ref{ambient-cor})
 and $p,q$ are general points.
We then lift $C_{00}$ to 
\eqspl{c10}{C_{10}=C_{11}\cup_{p,q}C_{12}\subset X_1}
using the surjection $\psi:G_{C_{00}}\to M_0$ to a line bundle of degree $e$
of the form $\O_{C_0}(H+A_0)$ as above.
We choose the line bundle $M_0$ on $C_{00}$  so that
\[e_1:=\deg(M_0|_{C_{01}})\equiv d(d-1)\mod n-d, e_1\geq m, e_2:=\deg(M_0|_{C_{02}})\geq (g-1)n\]
and \[e_1+e_2=e.\] 
We may assume $e_1\leq 2n$.
Now we have analogues of the sequence \eqref{normalseq}
for $C_{11}, C_{12}$ and inductively both left and right members
in those sequences have Euler slope $\geq 2$, and it follows that
\[H^1(N_{C_{1i}/X_1}(-p-q))=0, i=1,2.\]
Because $N_{C_{10}/X_1}$ contains $N_{C_{11}/X_1}(-p-q)\oplus N_{C_{12}/X_1}(-p-q)$
as a subsheaf parametrizing deformations where $C_{11}$ and $C_{12}$ deform separately
going through $p,q$, it follows easily that
$C_{10}$ is smoothable in $X_1$ to a curve of genus g and degree $e=e_1+e_2$. 
Now the bundle $K^*(M)$ restricts to the analogous bundles
on $C_{1i}, i=1,2$ which are balanced by induction
and perfect for $i=1$ by the congruence condition on $e_1$. Moreover as noted the Euler slope of 
$K^*(M)|_{C_{11}}$ is clearly at least 2.
Hence by Lemma \ref{reducible} it follows that $K^*(M)$ is balanced on $C_{10}$, hence on its smoothing in $X_1$.
\par
Finally for ambient-balancedness, 
we argue as in the proof of Theorem \ref{bmb-antican-thm}, 
noting that here again $C_2$ is a union of lines $L$ with trivial normal
bundle, hence 
\[T_{X_2}\mlog{E}|_L=\O(1)\oplus(n-2)\O_L\] where
the $(n-2)\O_L$ quotient coincides at $p=L. C_1$ with $T_{p,E}$. 
Moreover $C_2\cap C_1=A$ is a general divisor on $C_1$.
As in the above proof, it will
suffice to prove that $T_{X_1}|_{C_1}$ is balanced.
Note the exact sequence
\[\exseq{K^*(M)}{T_{\P(G)}|_{C_1}}{T_{\P^m}|_{C_1}},\]
which identifies $K^*(M)$ as the relative tangent bundle $T_{X_1/\P^m}$.
\par
%************
%\par, and similarly
%\[\exseq{T_{X_1/\P^m}\mlog{E}}{T_{X_1}\mlog{E}|_{C_1}}{T_{\P^m}|_{C_1}}\]
%where $E=X_1\cap X_2$ is the exceptional divisor on $X_1=\P(G)$,
%so $T_{X_1/\P^m}\mlog{E}$ is the elementary  down modification of $K^*(M)$ 
%corresponding to $T_E$ at $E\cap C_1$. In the notation on \cite{caudatenormal}, proof of Lemma 33,
%$T_{X_1/\P^m}\mlog{E}$ coincides with $G_1^*(M)$ where $G_1$ fits in an exact sequence
%\[\exseq{\O(-mH)}{(n-m)\O}{G_1}\] so $\P(G_1)\subset\P(G)$ is the divisor corresponding to $\O(1)\to G$
%(cf. \eqref{g-seq}). Note the exact diagram on $C_1$
%\eqspl{}{
%\begin{matrix}
%	&&&&0&&0&&\\
%	&&&&\downarrow&&\downarrow&&\\
%	&&&&\O(H)&=&\O(H)&&\\
%	&&&&\downarrow&&\downarrow&&\\
%	0&\to&S_G&\to&G|_{C_1}&\to&M&\to&0\\
%	&&\parallel&&\downarrow&&\downarrow&&\\
%	0&\to&S_G&\to&G_1|_{C_1}&\to&M_A&\to&0\\
%	&&&&\downarrow&&\downarrow&&\\
%	&&&&0&&0&&
%		\end{matrix}	
%}
%Dualizing the bottom sequence and comparing with \eqref{log-proj}, we conclude that
%\[T_{X_1/\P^m}\mlog{E}|_{C_1}=G_1^*(M)|_{C_1}.\]\par
%We aim to show that $T_{X_1}\mlog{E}|_{C_1}$ is perfectly balanced, which will hold provided
%*****************
\par
Now \eqref{match-ambient} ensures that the slopes of $K^*(M)$ and $T_{\P^m}|_{C_1}$ 
have the same roundoff,
so by Lemma \ref{match} it will suffice to show $K^*(M)$ and $T_{\P^m}|_{C_1}$ are balanced.
As for $T_{\P^m}|_{C_1}$, it may be assumed balanced thanks to Corollary \ref{ambient-cor}.
 As for $K^*(M)|_{C_1}$, we will use induction on $g$.
First for $g=0$, it is proven in \cite{caudatenormal},  Lemma 33, p. +35, 
that $K|_{C_1}$ is balanced,
hence so is $K^*(M)|_{C_1}$.
Then the general case is proven by degeneration to $C_{10}=C_{11}\cup C_{12}$ 
similarly to the above where $K^*(M)|_{C_{11}}$ is perfect.
%, using the numerical conditions
%\eqref{match-ambient} to guarantee perfect balance.\par
%Finally, as in the proof of Theorem \ref{bmb-antican-thm}, 
%$T_{X_2}\mlog{E}$ is balanced on $C_2$, and it follows as there that
%$C_1\cup C_2$ smooths out to an ambient-balanced curve on the general fibre $X$.

\end{proof}
\begin{rem}
	The ultra version of the Matching Lemma \ref{match} is not known. Therefore neither is the
	ultra version of Theorem \ref{fano}
	\end{rem}
\begin{rem}
	There is a misprint in the proof of Lemma 33 in the journal version of \cite{caudatenormal} (p.+35, l.-11).
	The arxiv version is correct.
	\end{rem}
\vfill\eject
\bibliographystyle{amsplain}
%\bibliography{../mybib}
\bibliography{../mybib}\vskip1cm

\providecommand{\bysame}{\leavevmode\hbox to3em{\hrulefill}\thinspace}
\providecommand{\MR}{\relax\ifhmode\unskip\space\fi MR }
% \MRhref is called by the amsart/book/proc definition of \MR.
\providecommand{\MRhref}[2]{%
  \href{http://www.ams.org/mathscinet-getitem?mr=#1}{#2}
}
\providecommand{\href}[2]{#2}
\begin{thebibliography}{10}

\bibitem{alyang}
A.~{A}tanasov, {E.} Larson, and {D}. Yang, \emph{Interpolation for normal
  bundles of general curves}, Mem. AMS (2016), arxiv 1509.01724v3.

\bibitem{coskun-riedl}
I.~Coskun and E.~Riedl, \emph{Normal bundles of rational curves on complete
  intersections}, Math Z. \textbf{288} (2018), 803--827, arxiv 1705.08441v1.

\bibitem{ein-laz-normal}
L.~Ein and R.~Lazarsfeld, \emph{Stability and restrictions of {P}icard bundles,
  with an application to the normal bundles of elliptic curves}, Complex
  projective geometry (Trieste 1989/ Bergen 1989) (G.~Ellingsrud, C.~Peskine,
  G.~Sacchiero, and S.~A. Stromme, eds.), London Math. Society Lecture Note
  series, vol. 179, Cambridge University Press, 1992, pp.~149--156.

\bibitem{eisenbud-vandeven-normal}
D.~Eisenbud and A.~Van de~Ven, \emph{On the normal bundle of smooth rational
  space curves}, Math. Ann. \textbf{256} (1981), 453--463.

\bibitem{elligsrud-laksov}
{G}. {E}llingsrud and {D}. {L}aksov, \emph{The normal bundle of elliptic space
  curves of degree 5}, 18th Scandinavian congress of math. Proc. 1980
  (E.~Balslev, ed.), Birkh\"auser, 1981, pp.~258--287.

\bibitem{hulek-elliptic}
{K}. {H}ulek, \emph{Projective geometry of elliptic curves}, Algebraic
  Geometry- Open Problems, Lecture Notes in Math., vol. 997, Springer-Verlag,
  1983, pp.~228--266.

\bibitem{perrin}
D.~Perrin, \emph{Courbes passant par $m$ points g\'en\'eraux de
  $\mathbf{P}^3$}, M\'em. Soc. Math. France (N.S.) (1987), no.~28-29.

\bibitem{ran-normal}
Z.~Ran, \emph{Normal bundles of rational curves in projective spaces}, Asian J.
  Math. \textbf{11} (2007), 567--608.

\bibitem{scroll}
\bysame, \emph{Interpolation of rational scrolls}, arxiv.org (2021),
  2111.02466.

\bibitem{hypersurf}
\bysame, \emph{Low-degree rational curves on hypersurfaces in projective spaces
  and their fan degenerations}, J. Pure Applied Algebra (2021), Arxiv
  1906.03747.

\bibitem{caudatenormal}
\bysame, \emph{Balanced curves and minimal rational connectedness on {F}ano
  hypersurfaces}, Internat. Math. Research Notices (2022),
  arxiv.math:2008.01235.

\bibitem{sacchiero}
G.~Sacchiero, \emph{Fibrati normali di curve razionali dello spazio
  proiettivo}, Ann. Univ. Ferrara \textbf{VII (N. S.)} (1981), no.~26, 33--40.

\end{thebibliography}
UC Math Dept. \nl
Skye Surge Facility, Aberdeen-Inverness Road,\nl
Riverside CA 92521 US\nl 
\email{ziv.ran@  ucr.edu}
\vfill\eject
%\documentclass[12pt]{amsart}
%%\documentclass{letter}
%\pagestyle{empty}
%%\documentclass{amsbook}
%\usepackage{amsmath}
%\usepackage{amssymb}
%\usepackage{amsfonts}
%\usepackage{amsthm}
%\usepackage{verbatim}
%\usepackage{stackrel}
%
%\begin{document}

%\title{ \small{ }
%\footnote{2010 {\it Mathematics Subject
%Classification}.Primary 11B25.} \footnote{{\it Key words}. arithmetic progressions, quantitative %Nullstellensatz.}}
%\author{\small{Mei-Chu Chang}\footnote{Research partially
%financed by the NSF Grants~DMS~1764081.}\\ \texttt{\small{Department of Mathematics}}\\
%\texttt{\small{University of California, Riverside}}\\\texttt{\small
%mcc@math.ucr.edu}}

\section*{Appendix by M. C. Chang:\\ Some roundoff equations arising from degree arithmetic}
\medskip
%{\bf %Appendix:%$\qquad\qquad\qquad\qquad\qquad\qquad\qquad\qquad\qquad\qquad\qquad\qquad\qquad\qquad\qquad\qquad\qquad
%$\\Some roundoff equations arising from degree arithmetic}

\medskip

$\qquad\qquad\qquad\qquad\qquad\quad$
{{By Mei-Chu Chang}}
\footnote{Research partially
	financed by the NSF Grants~DMS~1764081.}\par
\hskip0.3cm  Department of Mathematics,
UC Riverside, Riverside CA 92521 mcc AT math.ucr.edu
%$\qquad\qquad\qquad\qquad\qquad\quad$Mei-Chu Chang\footnote{Research partially
%financed by the NSF Grants~DMS~1764081.}

\bigskip
\noindent In this appendix, we prove the following

\medskip

\noindent{\bf Theorem 1.} {\it For fixed integers $3\leq d\leq n-1 $ and $g>0$, there are at least $(d-2)(n-d+1)$ 
	arithmetic progressions with difference $ d(n-2)$ of  $e$ values such that for 
	some integer $e_0$, $e\geq e_0\geq (g+1)n$, one has

\noindent either
\begin{equation}\label{(1)}\bigg\lfloor \frac{-de_0+e}{n-d}\bigg\rfloor +e =e_0 +\bigg\lfloor\frac{2e_0+2g-2}{d-2}\bigg\rfloor.\end{equation}
or
\begin{equation}\label{(1.a)}\frac{-de_0+e}{n-d} +e =e_0 +\bigg\lfloor\frac{2e_0+2g-2}{d-2}\bigg\rfloor+1.\end{equation}
 }

\bigskip

\smallskip

\noindent Our approach is similar to that of the case $g=0$ of \eqref{(1)} 
(see \cite{caudatenormal}, Appendix
by M. C. Chang). So we only provide the necessary details here.

%\medskip

%\noindent For a more precise estimate of the number of congruence classes, see Remark 7.

\bigskip

%\smallskip

\noindent We write

\begin{equation}\label{(2)} g=x(d-2) +y, \text{ where }\,\, y\in [0, d-3],  \end{equation}
\noindent and denote
\begin{equation}\label{(3)}b=n-d+1. \qquad\qquad\qquad\qquad\qquad\quad  \end{equation}

\medskip

\noindent For $(c,r)\in [0, d-3]\times [0, b-1]$, and $k\in \mathbb Z_{+}$ let

\begin{equation}\label{(4)}e_0= (kb+r)(d-2)+c. \qquad\qquad\qquad \end{equation}

\noindent Hence

\begin{equation}\label{(5)}\frac{2e_0+2g-2}{d-2}=2kb+2r+2x+\frac{2c+2y-2}{d-2}. \end{equation}
\bigskip
\noindent Denote
\begin{equation}\label{(6)} t=\bigg\lfloor\frac{2c+2y-2}{d-2}\bigg\rfloor.\qquad\qquad\qquad\qquad   \end{equation}

\smallskip

\noindent
For equation \eqref{(1)}, we let $\varepsilon$ be the fractional part of $\frac{-de_0+e}{n-d}$,

\noindent i.e.,
\begin{equation}\label{(7)}\frac{-de_0+e}{n-d} = \bigg\lfloor \frac{-de_0+e}{n-d}\bigg\rfloor +\varepsilon. \end{equation}
In particular, $\varepsilon<1$.

\bigskip

\noindent Putting displays \eqref{(1)} and \eqref{(2)}-\eqref{(7)} together, we have

\begin{equation}\label{(8)}e=d(n-2)k+2x+t+rd+c+\frac {r(d^2-3d)+c(d-1)-2x-t}{b} + \varepsilon\frac{b-1}{b}. \end{equation}

\smallskip

\noindent For equation \eqref{(1.a)}, we have

\smallskip

\begin{equation}\label{(8.a)}e=d(n-2)k+2x+t+1+rd+c+\frac {r(d^2-3d)+c(d-1)-2x-t-1}{b}. \end{equation}

\bigskip

%\noindent {\it Since we consider the lower bound on the number of $e$'s, we may assume that
%$c > 0.$}

\bigskip

\noindent {\bf Lemma 2.} {\it Let
$$\begin{aligned}&e(c,r,\epsilon)=e\\=&d(n-2)k+2x+t+rd+c+\frac {r(d^2-3d)+c(d-1)-2x-t}{b} + \varepsilon\frac{b-1}{b}\end{aligned}$$
\noindent be as in \eqref{(8)}

\smallskip

\noindent If $(c,r,\epsilon)\not=(c_1, r_1, \epsilon_1)$, then  $e(c,r, \epsilon)\not\equiv e(c_1, r_1,  \epsilon_1)\mod d(n-2).$

\medskip

\noindent The same statement is also true for $e$ in} \eqref{(8.a)}.

\bigskip

\noindent{\it Proof.} Let $$E(c,r,\epsilon)=2x+t+rd+c+\frac {r(d^2-3d)+c(d-1)-2x-t}{b} + \varepsilon\frac{b-1}{b}.$$

\medskip

\noindent {\it Claim.1.  $E(c,r,\epsilon)\not=E(c_1,r_1,\epsilon_1)$ as real numbers.}

\medskip

\noindent {\it Proof of Claim.1}.

\smallskip

\noindent
First, we assume $r_1-r\geq 1$, and $ E(c,r,\epsilon)=E(c_1,r_1,\epsilon_1)$. Then
\begin{equation}\label{9}(r_1-r)\bigg( d+\frac {d^2-3d}b\bigg)=(c-c_1)\bigg(1+\frac {d-1}b\bigg)+(t-t_1)\bigg(1-\frac 1b\bigg)+\frac {b-1}b (\epsilon-\epsilon_1). \end{equation}
By Lemma 4 below, $t-t_1\leq 2$. Also, in \eqref{(4)}, we take  $c\in[0, d-3]$. Hence, the right hand side of (9) is less than
$$(d-3)\,\frac {b+d-1}b+2\,\frac{b-1}b+\frac {b-1}b,$$
which is less than
$$ d+\frac {d^2-3d}b-\frac db< d+\frac {d^2-3d}b\leq {\rm the}\, {\rm  left}\, {\rm  hand}\, {\rm  side} \,{\rm of}\, (9).$$
This is a contradiction.

\noindent
Hence, $r_1=r$ and (9) is
\begin{equation}\label{(10)}0=(c-c_1)\bigg(1+\frac {d-1}b\bigg)+(t-t_1)\bigg(1-\frac 1b\bigg)+\frac {b-1}b (\epsilon-\epsilon_1). \end{equation}

\noindent Next, we assume $c-c_1\geq 1$. From the definition of $t$ in \eqref{(6)}, we have $t\geq t_1$, and
$$b+d-1\leq (c-c_1)(b+d-1)+(t-t_1)(b-1)=(b-1)(\epsilon_1-\epsilon)\leq b-2,$$
which is a contradiction.

\medskip

\noindent {\it Claim.2.  If $E(c_1,r_1,\epsilon_1)\not=E(c,r,\epsilon)$, then $$E(c_1,r_1,\epsilon_1)\not\equiv E(c,r,\epsilon) \mod d(n-2).$$}

\noindent {\it Proof of Claim.2.} Assume $E(c_1,r_1,\epsilon_1)>E(c,r,\epsilon)$. Then
$$\begin{aligned} E(c_1,r_1,\epsilon_1)-E(c,r,\epsilon)< & (b-1)\,\frac{bd+d^2-3d}b+ (d-3)\,\frac {b+d-1}b+2\,\frac{b-1}b+\frac {b-1}b\\
= &bd+d^2-3d-\frac{d}b\\
<&d(n-2).\end{aligned}$$
Hence $E(c_1,r_1,\epsilon_1)-E(c,r,\epsilon)$ cannot be a multiple of $d(n-2). \qquad\square$

\bigskip

\bigskip

\noindent {\it Proof of Theorem 1}.

\medskip

\noindent By Lemma 2, it is sufficient to find a lower bound on the permissible pairs of $(c,r)\in [0,d-3]\times[0,b-1]$.

\bigskip
\noindent Since $e\in \mathbb Z$, in \eqref{(8)} we let
$$\frac {r(d^2-3d)+c(d-1)-2x-t}{b} + \varepsilon\frac{b-1}{b}=m+1, \, \text { with } m\in \mathbb Z.$$

\noindent By display \eqref{(7)}, $\epsilon <1$, which is equivalent to
$$m< \frac {r(d^2-3d)+c(d-1)-2x-t-1}{b}.$$

\smallskip

\noindent
So, for equation \eqref{(1)}, we want to rule out those $(c, r)$ such that
$$ \frac {r(d^2-3d)+c(d-1)-2x-t-1}{b}=m.$$

\smallskip

\noindent
Namely, we want to rule out $(c, r) \in [0,d - 3] \times [0,b - 1]$ such that

\begin{equation}\label{(11)}r(d^2-3d)+c(d-1)\equiv 2x+t+1\mod b.\end{equation}
\smallskip

\noindent On the other hand, these $(c,r)\in [0,d - 3] \times [0,b - 1]$ satisfying \eqref{(11)} are the pairs making $e$ in \eqref{(8.a)} an integer, hence a solution of \eqref{(1.a)}. Therefore, $\forall(c,r)\in [0,d - 3] \times [0,b - 1]$ is a permissible pair for either equation \eqref{(1)} or equation \eqref{(1.a)}. $\qquad\square$

%
%
%\bigskip
%
%\bigskip
%
%\centerline {REFERENCES}
%
%\bigskip
%\noindent 1. M.-C. Chang, Appendix in Z. Ran, Balanced curves and minimal rational connectedness on Fano hypersurfaces, Internat. Math. Research
%Notices (2022), arxiv.math:2008.01235.
%
%\bigskip
%
%\bigskip
%
%\noindent Department of Mathematics
%
%\noindent University of California, Riverside
%
%\noindent mcc@math.ucr.edu
%\end{document}

\vfill\eject

%\bibliography{./mybib}
\end{document}